\documentclass[12pt]{amsart}
\usepackage{amssymb}
\usepackage{eucal}
\usepackage[bookmarks=false]{hyperref}

\allowdisplaybreaks

\newtheorem{theorem}{Theorem}[section]
\newtheorem{lemma}[theorem]{Lemma}
\newtheorem{proposition}[theorem]{Proposition}
\newtheorem{corollary}[theorem]{Corollary}

\newtheorem*{theorem:pirho}{Theorem~\ref{T:pirho}}
\newtheorem*{theorem:Lvalue1}{Theorem~\ref{T:Lvalue1}}
\newtheorem*{theorem:Lambdavalue1}{Theorem~\ref{T:Lambdavalue1}}
\newtheorem*{corollary:Lambdaclassno}{Corollary~\ref{C:Lambdaclassno}}

\theoremstyle{definition}

\newtheorem{example}[theorem]{Example}
\newtheorem*{acknowledgments}{Acknowledgments}

\theoremstyle{remark}
\newtheorem{remark}[theorem]{Remark}

\newcommand{\C}{\ensuremath \mathbb{C}}
\newcommand{\Z}{\ensuremath \mathbb{Z}}
\newcommand{\LL}{\mathbb{L}}
\newcommand{\BB}{\mathbb{B}}

\newcommand{\F}{\ensuremath \mathbb{F}}

\newcommand{\TT}{\mathbb{T}}

\newcommand{\bA}{\mathbf{A}}
\newcommand{\bH}{\mathbf{H}}
\newcommand{\bK}{\mathbf{K}}

\newcommand{\cE}{\mathcal{E}}
\newcommand{\cF}{\mathcal{F}}

\newcommand{\cL}{\mathcal{L}}
\newcommand{\cO}{\mathcal{O}}
\newcommand{\cQ}{\mathcal{Q}}
\newcommand{\cR}{\mathcal{R}}
\newcommand{\cS}{\mathcal{S}}
\newcommand{\cT}{\mathcal{T}}
\newcommand{\cU}{\mathcal{U}}

\newcommand{\fa}{\mathfrak{a}}
\newcommand{\fb}{\mathfrak{b}}
\newcommand{\fc}{\mathfrak{c}}
\newcommand{\fM}{\mathfrak{M}}
\newcommand{\fp}{\mathfrak{p}}
\newcommand{\fq}{\mathfrak{q}}

\newcommand{\inv}{\ensuremath ^{-1}}

\DeclareMathOperator{\Cl}{Cl}

\DeclareMathOperator{\divisor}{div}

\DeclareMathOperator{\Ext}{Ext}
\DeclareMathOperator{\Fr}{Fr}
\DeclareMathOperator{\Gal}{Gal}

\DeclareMathOperator{\ord}{ord}

\DeclareMathOperator{\Res}{Res}
\DeclareMathOperator{\sgn}{sgn}
\DeclareMathOperator{\Span}{Span}
\DeclareMathOperator{\Spec}{Spec}

\newcommand{\twist}{^{(1)}}
\newcommand{\invtwist}{^{(-1)}}

\newcommand{\twisti}{^{(i)}}
\newcommand{\twistk}[1]{^{(#1)}}

\newcommand{\oa}{\overline{a}}
\newcommand{\oalpha}{\overline{\alpha}}
\newcommand{\ob}{\overline{b}}
\newcommand{\obeta}{\overline{\beta}}
\newcommand{\oc}{\overline{c}}
\newcommand{\ogamma}{\overline{\gamma}}
\newcommand{\osigma}{\overline{\sigma}}
\newcommand{\oK}{\mkern2.5mu\overline{\mkern-2.5mu K}}

\newcommand{\iso}{\stackrel{\sim}{\to}}
\newcommand{\power}[2]{{#1 [[ #2 ]]}}
\newcommand{\laurent}[2]{{#1 (( #2 ))}}
\newcommand{\Ehat}{\widehat{E}}
\newcommand{\pd}{\partial}
\newcommand{\tchi}{\widetilde{\chi}}
\newcommand{\tE}{\widetilde{E}}
\newcommand{\tlambda}{\widetilde{\lambda}}
\newcommand{\tLambda}{\widetilde{\Lambda}}
\newcommand{\tf}{\widetilde{f}}
\newcommand{\tg}{\widetilde{g}}
\newcommand{\tpi}{\widetilde{\pi}}
\newcommand{\ttt}{\widetilde{t}}
\newcommand{\ty}{\widetilde{y}}
\newcommand{\tsgn}{\widetilde{\sgn}}

\def\XXint#1#2#3{{\setbox0=\hbox{$#1{#2#3}{\int}$ }
\vcenter{\hbox{$#2#3$ }}\kern-.6\wd0}}

\begin{document}

\title[Special $L$-values and shtuka functions]{Special $L$-values and shtuka functions for \\ Drinfeld modules on elliptic curves}

\author{Nathan Green}
\address{Department of Mathematics, Texas A{\&}M University, College Station,
TX 77843, USA}
\email{jaicouru@gmail.com}

\author{Matthew A. Papanikolas}
\address{Department of Mathematics, Texas A{\&}M University, College Station,
TX 77843, USA}
\email{papanikolas@tamu.edu}

%   General info
\thanks{This project was partially supported by NSF Grant DMS-1501362}


\date{October 7, 2017}

\begin{abstract}
We make a detailed account of sign-normalized rank~$1$ Drinfeld $\mathbf{A}$-modules, for~$\mathbf{A}$ the coordinate ring of an elliptic curve over a finite field, in order to provide a parallel theory to the Carlitz module for $\mathbb{F}_q[t]$.  Using precise formulas for the shtuka function for $\mathbf{A}$, we obtain a product formula for the fundamental period of the Drinfeld module.  Using the shtuka function we find identities for deformations of reciprocal sums and as a result prove special value formulas for Pellarin $L$-series in terms of an Anderson-Thakur function.  We also give a new proof of a log-algebraicity theorem of Anderson.
\end{abstract}


\maketitle

\tableofcontents

\section{Introduction} \label{S:Intro}

We let $\theta$ and $t$ be independent variables over a finite field $\F_q$ of $q$ elements.  Based on the notion of rigid analytic trivialization of Anderson~\cite{And86}, Anderson and Thakur~\cite[\S 2]{AndThak90} introduced the function
\begin{equation} \label{omegaC}
  \omega_C = (-\theta)^{1/(q-1)} \prod_{i=0}^\infty \biggl( 1 - \frac{t}{\theta^{q^i}} \biggr)^{-1},
\end{equation}
which converges in the Tate algebra $\TT$ of rigid analytic functions in $t$ on the closed unit disk of $\C_\infty$, where $\C_\infty$ is the completion of an algebraic closure of the Laurent series field $\laurent{\F_q}{1/\theta}$.  The function $\omega_C$ is meromorphic on all of $\C_\infty$ in the rigid analytic sense, and it satisfies the difference equation
\begin{equation} \label{omegaCdiffeq}
  \omega_C^{(1)} - (t-\theta)\omega_C = 0,
\end{equation}
where the twisting $\omega_C^{(1)}$ is defined to be the element of $\TT$ obtained by raising the coefficients of $\omega_C$, as a power series in $t$, to the $q$-th power.  Anderson and Thakur~\cite{AndThak90} showed that $\omega_C$ is central to the theory of the Carlitz module and its tensor powers, and moreover,
\begin{equation} \label{Carlitzper}
\Res_{t=\theta}(\omega_C) = -\tpi = (-\theta)^{q/(q-1)} \prod_{i=1}^\infty \Bigl(
1 - \theta^{1-q^i} \Bigr)^{-1},
\end{equation}
where $\tpi$ is the period of the Carlitz module.  Sinha~\cite[\S 4]{Sinha97} demonstrated that the theory of $\omega_C$ could be extended to certain $t$-modules with complex multiplication and that, more generally, logarithms on these $t$-modules could be obtained essentially through Anderson generating functions (see \S\ref{S:AndGenFun}, and also \cite{EP14}, \cite[\S 4.2]{Pellarin08}).

More recently Pellarin~\cite{Pellarin12} introduced a new class of $L$-series for $\F_q[t]$.  We define a quasi-character $\chi : \F_q[\theta] \to \F_q[t]$ by setting $\chi(a) = a(t)$, where $a(t)$ means the polynomial $a\in \F_q[\theta]$ evaluated at $\theta=t$, and then for $s \in \Z_+$ we set
\[
  L(\F_q[t]; s) = \sum_{a \in \F_q[\theta]_+} \frac{\chi(a)}{a^s} = \sum_{a \in \F_q[\theta]_+} \frac{a(t)}{a(\theta)^s}.
\]
Pellarin proved a number of remarkable facts about $L(\F_q[t];s)$, including that for fixed $s$ the function $L(\F_q[t];s)$ is an element of $\TT$ and extends to an entire function on $\C_\infty$.  He proved a special value formula at $s=1$ in terms of $\omega_C$~\cite[Thm.~1]{Pellarin12}, namely
\begin{equation} \label{Pellarinformula}
  L(\F_q[t];1) = -\frac{\tpi}{(t-\theta) \omega_C},
\end{equation}
which serves to interpolate values of Goss $L$-series of Dirichlet type as well as Carlitz zeta values at some negative integers.  There has been an abundance of research in recent years on Pellarin $L$-series, including special value formulas and multivariable generalizations (e.g., see~\cite{AnglesNgoDacRibeiro16b}, \cite{AnglesPellarin14}--\cite{AnglesPellarinRibeiro16}, \cite{Goss13}, \cite{PellarinPerkins16}--\cite{Perkins14b}).

In the present paper we undertake a comprehensive study of sign normalized rank~$1$ Drinfeld $\bA$-modules for rings $\bA$ that are coordinate rings of elliptic curves over finite fields.  We let $E$ be an elliptic curve over~$\F_q$ and let $\bA= \F_q[t,y]$ be its coordinate ring via a cubic Weierstrass equation.  For an isomorphic copy of $\bA$, which we label $A = \F_q[\theta,\eta]$ with fraction field $K = \F_q(\theta,\eta)$, we consider a particular rank~$1$ Drinfeld $\bA$-module
\[
  \rho : \bA \to H[\tau],
\]
where $H$ is the Hilbert class field of $K$ and $H[\tau]$ is the ring of twisted polynomials in the $q$-th power Frobenius endomorphism~$\tau$.  The Drinfeld module $\rho$ is sign-normalized in the sense of Hayes \cite{Hayes79}, \cite{Hayes92}, thus definable over $H$.  The construction of $\rho$, due to Thakur~\cite{Thakur93} and based on work of Drinfeld and Mumford (see \cite[Ch.~6]{Goss}, \cite{Mumford78}), is rooted in the shtuka function~$f$ for~$E$.  We have a natural point $\Xi = (\theta,\eta) \in E(K)$, and we can find a point $V=(\alpha,\beta) \in E(H)$ so that the function $f \in H(t,y)$ has divisor on $E$ given by
\[
  \divisor(f) = (V^{(1)}) - (V) + (\Xi) - (\infty),
\]
where $V^{(1)}$ is the image of $V$ under the $q$-th power Frobenius.  We specify that $V$ be chosen to be in the formal group of $E$ at the infinite place of $K$ and that $f$ have sign~$1$ (see~\S\ref{S:Notation} for details on signs), and in this way $f$ is uniquely determined.  In \eqref{fdef} we write $f = \nu/\delta$ for explicit $\nu$, $\delta \in H[t,y]$.  See \S\ref{S:DrinfeldModules} for the precise ways one uses $f$ to construct $\rho$, as well as its exponential and logarithm functions $\exp_\rho$ and $\log_\rho$.

To determine the period lattice $\Lambda_\rho = \ker(\exp_\rho)$ contained in $\C_\infty$ we construct $\omega_\rho$, which is an analogue of the Anderson-Thakur function~$\omega_C$, via the infinite product
\[
  \omega_\rho = \xi^{1/(q-1)} \prod_{i=0}^\infty \frac{\xi^{q^i}}{f^{(i)}},
\]
where $f^{(i)}$ is the $i$-th Frobenius twist of $f$ (see \S\ref{S:Notation}) and $\xi = -(m + \beta/\alpha)$ (see \eqref{mquotients} for the definition of $m$).  The product for $\omega_\rho$ converges in the Tate algebra $\TT[y]$, and it extends to an entire function on $E \setminus \{\infty\}$.  By comparison to~\eqref{omegaCdiffeq} we find that
\begin{equation} \label{omegarhotwistintro}
  \omega_\rho^{(1)} - f \cdot \omega_\rho = 0,
\end{equation}
and using the theory of Anderson generating functions we prove the following theorem.

\begin{theorem:pirho}
We have $\Lambda_\rho = A \pi_\rho$.  Setting $\xi = -(m + \beta/\alpha)$ and $\lambda = dt/(2y+a_1t+a_3)$,
\[
  \pi_\rho = -\!\Res_{\Xi} (\omega_\rho \lambda) = - \frac{\xi^{q/(q-1)}}{\delta^{(1)}(\Xi)} \prod_{i=1}^\infty \frac{\xi^{q^i}}{f^{(i)}(\Xi)}.
\]
\end{theorem:pirho}

One of the subtle problems in proving Theorem~\ref{T:pirho}, which does not arise in the $\F_q[t]$ case, is the need to work around the pole of~$f$ at~$V$.  Indeed in the $\F_q[t]$ case the shtuka function is simply $t-\theta$, whereas here $f$ is a bit more complicated (see \eqref{fdef}).  Work of Sinha~\cite{Sinha97} shows ways to overcome these problems, but we use Anderson generating functions to devise a more streamlined approach of identifying logarithms of algebraic points in terms of residues of functions satisfying Frobenius difference equations (see Theorem~\ref{T:varepstheorem}), of which Theorem~\ref{T:pirho} is a special case.  A similar approach for defining such functions $\omega_\rho$ for all genera has been devised independently by Angl\`es, Ngo Dac, and Tavares Ribeiro~\cite{AnglesNgoDacRibeiro16a}.

We define two types of Pellarin $L$-series for $\bA$, the first of which is
\[
  L(\bA;s) = \sum_{a \in A_+} \frac{\chi(a)}{a^s} = \sum_{a \in A_+} \frac{a(t,y)}{a(\theta,\eta)^s},
\]
where $A_+$ denotes the elements of $A$ of sign~$1$.  For fixed $s \in \Z_+$, the sum $L(\bA;s)$ converges in $\TT[y]$.  We then prove the following result, which provides a version of~\eqref{Pellarinformula} for the value of $L(\bA;s)$ at $s=1$.

\begin{theorem:Lvalue1}
As elements of $\TT[y]$,
\[
  L(\bA;1) = -\frac{\delta^{(1)}\, \pi_\rho}{f \omega_\rho}.
\]
\end{theorem:Lvalue1}

For the second type of Pellarin $L$-series, we recall an extension of $\rho$ to integral ideals of $A$ due to Hayes~\cite{Hayes79} (see \S\ref{S:LSeries}), where for non-zero $\fa \subseteq A$, we have $\rho_\fa \in H[\tau]$, and we set
\[
  \chi(\fa) = \frac{\rho_{\fa}(\omega_\rho)}{\omega_\rho},
\]
which is shown to be an element of $H(t,y)$ (see Lemma~\ref{L:chiprops}) and extends $\chi$ above.  We set
\[
  \LL(\bA;s) = \sum_{\fa \subseteq A} \frac{\chi(\fa)}{\pd(\rho_{\fa})^s},
\]
where $\pd(\rho_{\fa})$ is the constant term of $\rho_{\fa}$ with respect to $\tau$ and where by convention we sum only over $\fa \neq 0$.  We note that historically $\pd(\rho_{\fa})$ has been difficult to compute, even when knowing $\rho$ exactly (cf.\ Hayes~\cite[\S 10--11]{Hayes79}), but the utility of the shtuka function leads to precise formulas (see Lemmas~\ref{L:rhofp} and~\ref{L:chiprops}).
We then prove an extended version of Theorem~\ref{T:Lvalue1}.

\begin{theorem:Lambdavalue1}
As elements of $\TT[y]$,
\[
  \LL(\bA;1) = -\sum_{\sigma \in \Gal(H/K)} \biggl( \frac{\delta^{(1)}}{f} \biggr)^{\sigma} \cdot \frac{\pi_\rho}{\omega_\rho},
\]
where $H$ is the Hilbert class field of $K$.
\end{theorem:Lambdavalue1}

Using different methods, more recently Angl\`{e}s, Ngo Dac, and Tavares Ribeiro~\cite{AnglesNgoDacRibeiro16a} have also shown that similar sums are rational multiples of $\pi_{\rho}/\omega_{\rho}$, but although their arguments work for all genera, they do not obtain precise formulas.

By evaluating at $\Xi$ we obtain the following corollary that relates the value of $\LL(\bA;1)$ to the class number of $\bA$ itself (as elements of $\F_q$).

\begin{corollary:Lambdaclassno}
Let $h(\bA) = \#E(\F_q)$ be the class number of $\bA$. Then
\[
\LL(\bA;1)\big|_{\Xi} = h(\bA).
\]
\end{corollary:Lambdaclassno}

The primary tools for proving Theorems~\ref{T:Lvalue1} and~\ref{T:Lambdavalue1} are results on interpolations of reciprocal sums in \S\ref{S:sums}, which generalize work of Angl\`{e}s, Pellarin, and Simon~\cite{AnglesPellarin14}, \cite{AnglesPellarin15}, \cite{AnglesSimon}, and of Thakur~\cite{Thakur92}.  Namely for $i \geq 0$, we consider sums
\[
  S_i = \sum_{a \in A_{i+}} \frac{1}{a}, \quad S_{\fp,i} = \sum_{a \in \fp_{i+}} \frac{1}{a},
\]
where $\fp \subseteq A$ is a prime ideal of degree~$1$ corresponding to an $\F_q$-rational point on~$E$ and where $A_{i+}$, $\fp_{i+}$ denote subsets of elements of degree~$i$ and sign~$1$.  We consider also deformations in $t$, $y$ given by functions on $E$,
\[
  \cS_{i} = \sum_{a \in A_{i+}} \frac{a(t,y)}{a(\theta,\eta)}, \quad
  \cS_{\fp,i} = \sum_{a \in \fp_{i+}} \frac{a(t,y)}{a(\theta,\eta)} \quad \in K[t,y].
\]
We apply the methods of Angl\`{e}s, Pellarin, and Simon to find exact formulas for each of these quantities.  For example, in Proposition~\ref{P:SiSpi} we find for $i \geq 2$,
\begin{equation}
  \cS_i = S_i \cdot \frac{g_i}{\nu^{(i-1)}}\cdot f f^{(1)} \cdots f^{(i-1)},
\end{equation}
where $\nu$ is the numerator of the shtuka function in~\eqref{fdef} and $g_i$ is given by a specific linear polynomial in $t$, $y$ in~\eqref{giformula}.  We further determine a formula for $S_i$ itself in Theorem~\ref{T:Siformula}, namely for $i \geq 2$,
\begin{equation} \label{Siformulaintro}
  S_i = \frac{\nu^{(i)}}{g_i^{(1)} \cdot f^{(1)} \cdots f^{(i)}} \Bigg|_{\Xi},
\end{equation}
which provides a functional interpretation of a previous formula of Thakur~\cite[Thm.~IV]{Thakur92}.  We should also compare~\eqref{Siformulaintro} with older results of Carlitz~\cite[Eq.~(9.09)]{Carlitz35} on reciprocal sums for $\F_q[\theta]$ (see \eqref{Carlitzsum}).  Similar formulas for $S_{\fp,i}$ and $\cS_{\fp,i}$ are obtained in Proposition~\ref{P:SiSpi} and Theorem~\ref{T:Siformula}.

In \cite{And94}, \cite{And96}, Anderson introduced the notion of log-algebraic power series identities for Drinfeld modules of rank~$1$ and connected these identities to the theory of Goss $L$-series and special zeta values.  The results on reciprocal sums and their interpolations from \S\ref{S:sums} also yield a new function theoretic proof of a theorem of Anderson~\cite[Thm.~5.1.1]{And94} in \S\ref{S:LogAlg} (see Theorem~\ref{T:Anderson}). The approach is similar to a proof of Thakur~\cite[\S 8.10]{Thakur} for the $\F_q[t]$ case, but it is somewhat more involved due to the intricacies of the formulas for $S_i$ and $S_{\fp,i}$ in Theorem~\ref{T:Siformula}.  Nevertheless, as a result we obtain precise information about the special polynomials in Anderson's theorem in terms of decompositions via the shtuka function (see Remark~\ref{Rem:Ebz}).  It would be interesting to work out the connections between these identities and the corresponding class modules of Taelman~\cite{Taelman12} and more recently of Angl\`es, Ngo Dac, and Tavares Ribeiro~\cite{AnglesNgoDacRibeiro17}, though we do not pursue an investigation here.  In \S\ref{S:Examples} we provide examples of the various techniques found in the paper.

One of our overriding goals has been to develop a theory for $\rho$ that is as explicit as the one now well-established for the Carlitz module \cite[Ch.~3]{Goss}, \cite[Ch.~2]{Thakur}.  It is natural to ask whether the current methods can be extended to rings $\bA$ associated to curves of genus~$\geq 2$.  Thakur's construction of the Drinfeld module $\rho$ in~\cite{Thakur93} does not require genus~$1$, but in higher genus cases the point $V$ now becomes an effective divisor of degree equal to the genus, making the construction of the shtuka function $f$ more difficult.  Furthermore, here we make frequent use of the group law on the elliptic curve $E$, and we imagine in higher genus similar constructions will require additional study of the Jacobian of the curve.  It is also natural to consider how our techniques can be used to investigate tensor powers $\rho^{\otimes n}$ over $\bA$ as in~\cite{AndThak90}, which will be the subject of future work.

\begin{acknowledgments}
The authors thank B.~Angl\`{e}s, C.-Y.~Chang, D.~Goss, U.~Hartl, and F.~Pellarin for advice and comments on previous versions of this paper.  We also thank the referees for several helpful suggestions.
\end{acknowledgments}

\section{Setting and notation} \label{S:Notation}

Let $q$ be a fixed power of a prime $p$, and let $\F_q$ be the field with $q$ elements.  We let $E$ be an elliptic curve defined over $\F_q$, given by the equation
\begin{equation}\label{Eeq}
  E: y^2 + a_1 ty + a_3 y  = t^3 + a_2 t^2 + a_4 t + a_6, \quad a_i \in \F_q,
\end{equation}
with point at infinity set to be $\infty$.  We let $\bA = \F_q[t,y]$ be the coordinate ring of functions on $E$ regular away from $\infty$, and we let $\bK = \F_q(t,y)$ be its fraction field.  We let
\[
  \lambda = \frac{dt}{2y + a_1 t + a_3}
\]
be the invariant differential on $E$.

We also fix other variables $\theta$, $\eta$, with $\theta$ independent from $t$ and $\eta$ independent from $y$ over~$\F_q$ so that $A=\F_q[\theta,\eta]$ and $K = \F_q(\theta,\eta)$ are isomorphic
copies of $\bA$ and $\bK$, together with canonical isomorphisms,
\begin{equation}\label{caniso}
  \chi : K \to \bK, \quad \iota : \bK \to K.
\end{equation}
For $a\in K$ we will also write $\oa = \chi(a) \in \bK$.  We let $\ord_\infty$ denote the valuation of $K$ at the infinite place, and we set $\deg = -\ord_\infty$.  They are normalized so that
\[
  \deg(\theta) = 2, \quad \deg(\eta) = 3,
\]
and as such we have an absolute value on $K$ defined by $|g| = q^{\deg(g)}$, $g \in K$.  When needed we will also define $\ord_\infty$ and $\deg$ on $\bK$ in the same way.  We set $K_\infty$ to be the completion of $K$ at the infinite place, and we take $\C_\infty$ to be the completion of an algebraic closure of $K_\infty$.  We note that by design
\begin{equation}
  \Xi = (\theta,\eta)
\end{equation}
is a $K$-rational point of $E$ and further that $\Xi$ is an element of the formal group $\Ehat(\fM_{\infty})$, where $\fM_\infty$ is the maximal ideal of the valuation ring of $K_\infty$.

As an $\F_q$-vector space, $\bA$ has a basis
\begin{equation} \label{AFqbasis}
  \bA = \Span_{\F_q}(t^i, t^jy : i \geq 0, j \geq 0),
\end{equation}
and since the monomials listed have distinct degrees, we can define the notion of the leading term of a non-zero element $a \in \bA$.  From this we define a sign function $\sgn : \bA\setminus \{ 0 \} \to \F_q^{\times}$, by setting $\sgn(a) \in \F_q^\times$ to be the leading coefficient of~$a \in \bA \setminus \{ 0 \}$.  This definition then extends to a group homomorphism on the completion $\bK^{\times}_{\infty}$,
\[
  \sgn:\bK^{\times}_{\infty}\to \F_q^\times.
\]
We say that an element of $\bA$ (or more generally of $\bK_{\infty}$) is monic if it has sign $1$.  In exactly the same way we define $\sgn : K^{\times}_\infty \to \F_q^{\times}$ and monic elements of $A$:
\[
  A_+ = \{ a \in A \mid \sgn(a) = 1 \}, \quad
  A_{i+} = \{ a \in A_+ \mid \deg a = i \}, \ i \geq 0.
\]
Finally, if $L/\F_q$ is any field extension, then $L[t,y] = L \otimes_{\F_q} \bA$ is the coordinate ring of $E$ over~$L$, and we can define a group homomorphism
\[
  \tsgn : L(t,y)^{\times} \to L^{\times},
\]
which extends $\sgn$ on $\bK$, by setting $\tsgn(g) \in L^{\times}$ to be the leading coefficient of non-zero $g \in L[t,y]$, with respect to $t$ and $y$, and then extending to quotients.

Let $L/\F_q$ be an extension of fields, with $L$ algebraically closed.  Let $\tau : L \to L$ be the $q$-th power Frobenius, and let $L[\tau]$ be the ring of twisted polynomials in $\tau$, subject to the relation $\tau c = c^q \tau$ for $c \in L$.  For $\Delta = \sum c_i \tau^i \in L[\tau]$, we set $\pd(\Delta) = c_0$ to be its constant term.

We define the Frobenius twisting automorphism of $L[t,y]$ by
\[
  g \mapsto g^{(1)} : \sum_{i,j} c_{ij} t^i y^j \mapsto \sum_{i,j} c_{ij}^{q} t^i y^j.
\]
As usual for $i \in \Z$ we set $g^{(i)}$ to be the $i$-th iterate of $g \mapsto g^{(1)}$.
As an automorphism of $L(t,y)$, Frobenius twisting has fixed field $\bK$.  We also extend $L[\tau]$ to a ring of operators $L(t,y)[\tau]$, which for $g \in L(t,y)$ we have $\tau g = g^{(1)} \tau$.  In this way $L(t,y)$ becomes a left $L(t,y)[\tau]$-module by setting for $\Delta = \sum_i g_i \tau^i$,
\[
  \Delta(h) = \sum_i g_i h^{(i)}.
\]
If $X$ is a point on $E$, then we let $X^{(i)}=\Fr^i(X)$ where $\Fr: E \to E$ is the $q$-th power Frobenius isogeny, and we can then extend Frobenius twisting to divisors.  We see then that
\[
  \divisor\bigl( g^{(i)} \bigr) = \bigl( \divisor(g) \bigr)^{(i)}, \quad g \in L(t,y),\ i \in \Z.
\]
We will make frequent and often implicit use of the fact that a divisor on $E$ is principal if and only if the sum of the divisor is trivial on $E$~\cite[Cor.~III.3.5]{Silverman}.

We consider two Tate algebras,
\begin{equation} \label{Tatealgs}
  \TT = \biggl\{ \sum_{n=0}^\infty c_n t^n \in \power{\C_\infty}{t} \biggm| |c_n| \to 0 \biggr\}, \quad
  \TT_\theta = \biggl\{ \sum_{n=0}^\infty c_n t^n \in \power{\C_\infty}{t} \biggm| \big\lvert \theta^n c_n \big\rvert \to 0 \biggr\},
\end{equation}
where $\TT$ consists of functions that converge on the closed unit disk of $\C_\infty$ and likewise functions in $\TT_\theta$ converge on the closed disk of radius $|\theta|$.  We have natural embeddings $\bA \hookrightarrow \TT_\theta[y] \hookrightarrow \TT[y]$, where the variables $t$ and $y$ satisfy equation \eqref{Eeq}.   The ring $\TT$ is complete with respect to the Gauss norm $\lVert\, \cdot\, \rVert$, where for $g = \sum c_n t^n \in \TT$,
\[
  \lVert g \rVert = \max_n \bigl( \lvert c_n \rvert \bigr).
\]
The ring $\TT[y]$ is complete under the extension of $\lVert\, \cdot\, \rVert$ defined by $\lVert g + hy \rVert = \max(\lVert g \rVert, \lVert h \rVert)$ for $g$, $h \in \TT$.  In the sense of~\cite[Chs.~3--4]{FresnelvdPut}, the rings $\TT[y]$ and $\TT_\theta[y]$ are affinoid algebras corresponding to rigid analytic affinoid subspaces of $E/\C_\infty$.  In the spirit of Sinha~\cite[\S 4.2]{Sinha97} (see also \cite[\S 4.2]{BP02}), if we let $\cE$ denote the rigid analytic variety associated to $E$ and we let $\cU \subseteq \cE$ be the inverse image under $t$ of the closed disk in $\C_\infty$ of radius $|\theta|$ centered at~$0$, then~$\cU$ is the affinoid subvariety of $\cE$ associated to $\TT_\theta[y]$.  Frobenius twisting extends to both $\TT$ and $\TT[y]$, and their fraction fields.  The rings $\TT$ and $\TT[y]$ have $\F_q[t]$ and $\bA$ as their respective fixed rings under twisting (cf.~\cite[Lem.~3.3.2]{P08}).

\section{Drinfeld modules, $\bA$-motives, and dual $\bA$-motives}\label{S:DrinfeldModules}

Let $L/K$ be a field extension.  A Drinfeld $\bA$-module, or simply Drinfeld module, over $L$ is an $\F_q$-algebra homomorphism
\[
  \rho : \bA \to L[\tau],
\]
such that for all $a \in \bA$,
\[
  \rho_a = \iota(a) + b_1 \tau + \dots + b_n \tau^n.
\]
That is, we require $\pd(\rho_a) = \iota(a)$.
The rank $r$ of $\rho$ is the unique integer such that $n = r \deg a$ for all $a$.  For more information about Drinfeld modules, see \cite[Ch.~4]{Goss}, \cite[Ch.~2]{Thakur}.

Our main objects of study are Drinfeld modules of rank~$1$ that are sign normalized in the sense of Hayes~\cite{Hayes79} (or see \cite[\S 7.2]{Goss}), which we will call Drinfeld-Hayes modules.  That is,
\[
  \rho_a = \iota(a) + b_1 \tau + \dots + \sgn(a) \tau^{\deg a}.
\]
The construction we will use is due to Thakur~\cite{Thakur93} by way of the shtuka function, which we now define (see also \cite[\S 7.11]{Goss}, \cite[\S 7.7]{Thakur}).

The isogeny $1 - \Fr : E \to E$ is separable and $(1-\Fr)^* \lambda = \lambda$, and furthermore, it induces an isomorphism of formal groups $1-\Fr: \Ehat(\fM_\infty) \iso \Ehat(\fM_\infty)$ (see \cite[Cor.~III.5.5, Cor.~IV.4.3]{Silverman}).  Therefore, we can pick a unique point $V \in \Ehat(\fM_\infty)$ so that
\begin{equation}\label{Vdef}
  (1-\Fr)(V) = V - V^{(1)} = \Xi,
\end{equation}
and moreover, $(1-\Fr)^{-1}(\Xi) = \{ V + P \mid P \in E(\F_q) \}$.  If we set
\begin{equation}
  V = (\alpha,\beta),
\end{equation}
then $\deg(\alpha) = \deg(\theta) = 2$ and $\deg(\beta) = \deg(\eta) = 3$.  In order to determine the signs of the coordinates of $V$ we use the coordinate $z = -t/y$ on the formal group $\Ehat(\fM_\infty)$, and in terms of $z$ we have~\cite[Cor.~III.5.5, Cor.~IV.4.3]{Silverman}
\[
(1-\Fr)(z) = z + O(z^2) \in \power{\F_q}{z},
\]
and thus
\[
(1-\Fr)(z(V)) = z(V) + O(z(V)^2) = z(\Xi).
\]
Therefore, $\sgn(z(V))=-1$, since $\sgn(z(\Xi)) = -1$ and $\deg(z(V)^2) < \deg(z(V))$.  Then, switching back to $t$ and $y$ coordinates~\cite[\S IV.1]{Silverman}, we find that
\begin{equation} \label{sgnalphabeta}
\sgn(\alpha) = \sgn\left (\frac{1}{z(V)^2}\right ) = 1, \quad \sgn(\beta) =  \sgn\left (\frac{-1}{z(V)^3}\right ) = 1.
\end{equation}

Define $H = K(\alpha,\beta)$.  Then there is a unique function $f \in H(t,y)$ with $\tsgn(f)=1$ so that
\begin{equation}\label{divf}
  \divisor(f) = (V^{(1)}) - (V) + (\Xi) - (\infty).
\end{equation}
The function $f$ is called the shtuka function for $\bA$.  The points $V^{(1)}$, $-V$, and $\Xi$ are collinear, and we take $m$ to be the slope of the line connecting them:
\begin{equation} \label{mquotients}
 m = \frac{\eta-\beta^q}{\theta - \alpha^q} = \frac{\eta + \beta + a_1\alpha + a_3}{\theta - \alpha}
 = \frac{\beta^q + \beta + a_1\alpha + a_3}{\alpha^q - \alpha}.
\end{equation}
We then write
\begin{equation} \label{fdef}
  f = \frac{\nu(t,y)}{\delta(t)} = \frac{y - \eta - m(t-\theta)}{t-\alpha} = \frac{y + \beta + a_1\alpha + a_3 - m(t-\alpha)}{t-\alpha},
\end{equation}
where we see that
\begin{gather} \label{nudiv}
\divisor(\nu) = (V\twist) + (-V) + (\Xi) - 3(\infty), \\
\label{deltadiv}
\divisor(\delta) = (V) + (-V) - 2(\infty).
\end{gather}
This construction of the shtuka function $f$ is originally due to Thakur~\cite{Thakur93}, and the reader is directed to~\cite[\S 8.2]{Thakur} for an expanded treatment.

As in \cite{Thakur93}, we now define an $\bA$-motive $M$, which we will use to define a Drinfeld-Hayes module $\rho$.  For our purposes $M$ will simply be a module over a particular non-commutative ring, and we will not need the general theory of $\bA$-motives (see \cite{And86}, \cite{HartlJuschka16} for additional properties of $\bA$-motives).  Let $L/K$ be an algebraically closed field, and let $U = \Spec L[t,y]$ be the affine curve $(L \times_{\F_q} E) \setminus \{ \infty \}$.  We let
\[
  M = \Gamma(U, \cO_{E}(V))  = \bigcup_{i \geq 0} \cL((V) + i(\infty)),
\]
where $\cL((V) + i(\infty))$ is the $L$-vector space of functions $g$ on $E$ with $\divisor(g) \geq -(V) - i(\infty)$.  We see easily that
\[
  \divisor(f f^{(1)} \cdots f^{(i-1)} ) = (V^{(i)}) - (V) + (\Xi) + (\Xi^{(1)}) + \dots +
  (\Xi^{(i-1)}) - i(\infty),
\]
and so by the Riemann-Roch theorem
\[
  \cL((V) + i(\infty)) = \Span_L \bigl( 1, f, f f^{(1)}, \dots, f f^{(1)} \cdots f^{(i-1)} \bigr).
\]
We make $M$ into a left $L[t,y,\tau]$-module by setting
\[
  \tau g = f g^{(1)}, \quad g \in M,
\]
and we find that $M$ is a projective $L[t,y]$-module of rank~$1$ as well as a free $L[\tau]$-module of rank~$1$ with basis~$\{ 1 \}$.  For $a(t,y) \in \bA$, $\deg a = i$, we see that $a \in \cL((V) + i(\infty))$ and that we have an expression
\begin{equation} \label{ashtukadecomp}
  a(t,y) = b_0 + b_1 f + \dots + b_i f f^{(1)} \cdots f^{(i-1)}, \quad b_j \in L,\ b_i \neq 0.
\end{equation}
The coefficient $b_0$ equals $a(\theta,\eta)= a(\Xi)$, since the basis elements $f, \dots, ff^{(1)} \cdots f^{(i-1)}$ all vanish at $\Xi$.  Using this construction we thus define a function
\[
  \rho : \bA \to L[\tau], \quad a(t,y) \mapsto a(\theta,\eta) + b_1 \tau + \dots + b_i\tau^i.
\]

\begin{proposition}[{Thakur \cite[\S 0.3.5]{Thakur93}}]
The function $\rho : \bA \to L[\tau]$ is a Drinfeld-Hayes module defined over $H = K(\alpha,\beta)$.
\end{proposition}

Thakur notes that this construction is a special case of more general results of Drinfeld on shtukas (see \cite[Ch.~6]{Goss}) and can be proved using methods of Mumford~\cite{Mumford78}.  To see that~$\rho$ is sign normalized, we observe by~\eqref{ashtukadecomp} that
\[
  \sgn(a(t,y)) = \tsgn( b_i f f^{(1)} \cdots f^{(i-1)}) = b_i,
\]
since each of $f, \dots, f^{(i-1)}$ has sign~$1$.  Since each of the functions $f, \ldots, f^{(i-1)}$ is defined over~$H$, as is $a(t,y)$, it follows that each of the coefficients $b_j \in H$.  To find $\rho$ functorially as an $\bA$-module from $M$, one can argue as in Sinha~\cite[\S 3.1.8]{Sinha97} that we have an isomorphism of $\bA$-modules,
\[
  \rho(L) \cong \Ext^1_{L[t,y,\tau]}(M,L[t,y]),
\]
where $\rho(L)$ is a copy of $L$ with $\bA$-module structure given by $\rho$.  We provide a more direct construction of $\rho$, due to Anderson, using dual $\bA$-motives later in this section.

The Drinfeld module $\rho$ is completely determined by $\rho_t$ and $\rho_y$, which we denote as
\begin{align}\label{taction}
\rho_t &= \theta + x_1 \tau + \tau^2, &
t &= \theta + x_1 f + f f^{(1)},\\
\label{yaction}
\rho_y &= \eta + y_1\tau + y_2\tau^2 + \tau^3, &
y &= \eta + y_1 f + y_2 f f^{(1)} + f f^{(1)} f^{(2)}.
\end{align}
The values of $x_1$, $y_1$, $y_2$ are determined by the relation $\rho_t \rho_y = \rho_y\rho_t$, and in fact $y_1$ and $y_2$ are determined by $x_1$ alone, as we see in the following theorem.

\begin{theorem}[{Dummit-Hayes~\cite[Thm.~2]{DummitHayes94}, \cite[\S 14--15]{Hayes92}}]
\label{T:DummitHayes}
The coefficients $x_1$, $y_1$, $y_2$ satisfy the equations
\begin{align*}
\theta y_1 + x_1 \eta^q &= x_1\eta + y_1\theta^q, &
\theta y_2 + x_1y_1^q + \eta^{q^2} &= \eta+ y_1x_1^q + y_2 \theta^{q^2}, \\
\theta + x_1 y_2^q + y_1^{q^2} &= y_1 + y_2 x_1^{q^2} + \theta^{q^3}, &
x_1 + y_2^{q^2} &= y_2 + x_1^{q^3}.
\end{align*}
We have $K(x_1) = K(x_1,y_1,y_2)$, and $K(x_1)$ is the Hilbert class field of $K$.  That is, $K(x_1)$ is the maximal unramified abelian extension of $K$ in which $\infty$ splits completely.  Moreover, $x_1$, $y_1$, $y_2$ are all integral over $A$.
\end{theorem}

\begin{proof}
The four equations arise by equating coefficients of powers of $\tau$ in $\rho_t \rho_y=\rho_y\rho_t$.  The first two equations yield
\[
  y_1 = \frac{x_1(\eta^q-\eta)}{\theta^q-\theta}, \quad
  y_2 = \frac{\eta^{q^2} - \eta + x_1 y_1^q - y_1x_1^q}{\theta^{q^2}-\theta},
\]
and so $y_1$ and $y_2$ can be determined successively from $x_1$.  Since $\rho$ is a sign normalized module, it follows from Hayes~\cite[Thm.~15.6]{Hayes92} that the coefficients of $\rho$ generate the Hilbert class field of $K$ and are integral over $A$~\cite[\S 14]{Hayes79}.
\end{proof}

As $H \supseteq K(x_1)$, we see that $H$ contains the Hilbert class field.  In fact they are equal.

\begin{proposition} \label{P:mx1}
With notation as above,
\begin{align}
m &= y_2 - x_1^q, \label{my2x1q} \\
x_1 &= m + m^q + a_1. \label{x1mmq}
\end{align}
Moreover, $H = K(\alpha,\beta) = K(x_1)$, and so $H$ is the Hilbert class field of $K$.
\end{proposition}

\begin{proof}
We note from the definitions of $\rho_t$ and $\rho_y$ that
\[
  0 = \theta - t + x_1 f + f f^{(1)}, \quad
  0 = \eta - y + y_1 f + y_2 f f^{(1)} + f f^{(1)} f^{(2)}.
\]
If we let $g_t$ and $g_y$ denote the right-hand sides of these two equations, then
\[
  0 = g_y - f g_t^{(1)} - (y_2-x_1^q) g_t
   = \eta - y - (y_2 - x_1^q) (\theta-t) + (y_1 + t - \theta^q - x_1(y_2-x_1^q))f.
\]
Thus
\begin{equation} \label{fmotive}
f = \frac{y - \eta - (y_2-x_1^q)(t-\theta)}{t - \theta^q + y_1 - x_1(y_2-x_1^q)}.
\end{equation}
If we compare with \eqref{fdef}, we see that
\begin{equation} \label{mt0motive}
  m = y_2-x_1^q, \quad
  \alpha = \theta^q - y_1 + x_1(y_2-x_1^q).
\end{equation}
This proves \eqref{my2x1q}.  Moreover, we see that $m$, $\alpha \in K(x_1)$ by Theorem~\ref{T:DummitHayes}.  By \eqref{mquotients}, we see that $\beta$ is expressible in terms of $m$ and $\alpha$, so $\beta \in K(x_1)$.  This proves the last part of the statement.  To prove \eqref{x1mmq} we observe that if we expand
\[
  \rho_{y^2 + a_1ty + a_3y} - \rho_{t^3 + a_2t^2 + a_4t + a_6} = 0,
\]
then the coefficient of $\tau^5$ yields
\begin{equation}
  y_2^{q^3} + y_2 + a_1 - x_1^{q^4} - x_1^{q^2} - x_1 = 0.
\end{equation}
If we raise the last equation of Theorem~\ref{T:DummitHayes} to the $q$-th power and subtract, we find
\[
  x_1 = y_2 - x_1^q + y_2^q - x_1^{q^2} + a_1 = m + m^q + a_1,
\]
which completes the proof.
\end{proof}

As noted by Anderson and Thakur~\cite[\S 0.3]{Thakur93}, the exponential and logarithm functions of~$\rho$ are determined by the shtuka function~$f$.  We specialize the above discussion to the case $L = \C_\infty$.  The exponential function is the unique $\F_q$-linear power series, which we denote with reciprocal coefficients (which are nonzero by Theorem \ref{T:explog}) as
\[
  \exp_\rho(z) = \sum_{i=0}^\infty \frac{z^{q^i}}{d_i} \in \power{H}{z}, \quad d_0 = 1,
\]
satisfying the functional equation
\begin{equation}\label{expfunctionalequation}
\exp_\rho(\iota(a)z) = \rho_a(\exp_\rho(z)), \quad a \in \bA.
\end{equation}
Then $\exp_\rho : \C_\infty \to \C_\infty$ is entire and surjective, and we denote its kernel by $\Lambda_\rho$, which is a discrete $A$-submodule of $\C_\infty$ of projective rank~$1$ (see \cite[Ch.~4]{Goss}).  The logarithm
\[
  \log_\rho(z) = \sum_{i=0}^\infty \frac{z^{q^i}}{\ell_i} \in \power{H}{z}, \quad \ell_0=1,
\]
is the formal inverse of $\exp_\rho$, and it converges on a disk of finite radius in $\C_\infty$.  Using the functional equation of $\exp_\rho$, we see that $\log_\rho$ must satisfy,
\begin{equation}\label{logfunctionalequation}
  \iota(a)\log_\rho(z) = \log_\rho(\rho_a(z)).
\end{equation}
The quantities $d_i$ and $\ell_i$ are indeed non-zero, as we see in the following theorem and its corollary.

\begin{theorem} \label{T:explog}
The functions $\exp_\rho(z)$ and $\log_\rho(z)$ are given by the following formulas.
\begin{enumerate}
\item[(a)] \textup{(Thakur~\cite[Prop.~0.3.6]{Thakur93})}
\[
  \exp_\rho(z) = \sum_{i=0}^\infty \frac{z^{q^i}}{(f f^{(1)} \cdots f^{(i-1)})|_{\Xi^{(i)}}}.
\]
\item[(b)] \textup{(Anderson~\cite[Prop.~0.3.8]{Thakur93})} Let $\tlambda \in \Omega^1_{E/H}(-(V) + 2(\infty))$ be the unique differential $1$-form such that we have residue $\Res_{\Xi}(\tlambda^{(1)} / f) = 1$.  Then
\[
  \log_\rho(z) = \sum_{i=0}^\infty \Res_{\Xi} \biggl( \frac{ \tlambda^{(i+1)}}{f f^{(1)} \cdots f^{(i)}} \biggr) z^{q^i}.
\]
\end{enumerate}
\end{theorem}

Anderson's and Thakur's results work equally well for rings $\bA$ that arise from curves of higher genus.  In our genus~$1$ case, part~(b) of this theorem simplifies to a direct evaluation in terms of $\delta$ and $f$.

\begin{corollary}\label{C:logcoef}
The function $\log_\rho(z)$ has the expression
\[
  \log_\rho(z) = \sum_{i=0}^\infty \left ( \frac{ \delta^{(i+1)}}{\delta^{(1)} f^{(1)} \cdots f^{(i)}} \bigg|_{\Xi} \right ) z^{q^i}.
\]
\end{corollary}

\begin{proof}
We claim that
\[
  \tlambda = \delta \lambda = \frac{\delta\,dt}{2y + a_1 t + a_3}.
\]
Certainly $\delta\lambda \in \Omega^1_{E/H}(-(V) + 2(\infty))$ by \eqref{deltadiv}.  From~\eqref{taction} we see that
\begin{equation} \label{x1ftwist}
  \frac{t-\theta}{f} = x_1 + f^{(1)},
\end{equation}
and using that $t-\theta$ is a uniformizer at $\Xi$, we calculate the residue
\begin{equation} \label{lambdares}
  \Res_{\Xi} \biggl( \frac{\delta^{(1)} \lambda^{(1)}}{f} \biggr)
  = \frac{\delta^{(1)} (t-\theta)}{f (2y + a_1 t + a_3)} \bigg|_{\Xi}
  = \frac{(\theta - \alpha^q) \cdot (x_1 + f^{(1)})|_{\Xi}}{2\eta + a_1 \theta + a_3}.
\end{equation}
Using \eqref{mquotients}, \eqref{fdef}, and \eqref{x1mmq}, a reasonably straightforward calculation yields
\begin{equation}\label{alphaidentity}
  (\theta - \alpha^q)(x_1 + f^{(1)}(\Xi)) = 2 \eta + a_1\theta + a_3,
\end{equation}
and so the residue in \eqref{lambdares} is~$1$ and $\tlambda = \delta \lambda$ as claimed.  Thus we see that for $i \geq 0$,
\[
\Res_{\Xi} \biggl( \frac{ \tlambda^{(i+1)}}{f f^{(1)} \cdots f^{(i)}} \biggr)
= \frac{(t-\theta)}{(2y + a_1t + a_3)f}\bigg|_{\Xi} \cdot \frac{\delta^{(i+1)}}{f^{(1)} \cdots f^{(i)}} \bigg|_{\Xi}
= \frac{1}{\theta - \alpha^q} \cdot \frac{\delta^{(i+1)}}{f^{(1)} \cdots f^{(i)}} \bigg|_{\Xi},
\]
and the result follows since $\delta^{(1)} = t - \alpha^q$.
\end{proof}

We record for later use formulas for $d_i$, $\ell_i$, derived from the preceding results:
\begin{align} \label{diformula}
  d_i &= f f^{(1)} \cdots f^{(i-1)} \big|_{\Xi^{(i)}}, \quad i \geq 1, \\
\label{liformula}
  \ell_i &= \frac{\delta^{(1)}}{\delta^{(i+1)}} \cdot f^{(1)} \cdots f^{(i)}\Big|_{\Xi}, \quad i \geq 1.
\end{align}
In~\cite[\S I]{Thakur92}, Thakur defines quantities $f_i$, $g_i$ by $f_0=g_0=1$ and
\[
  f_i = \frac{d_i}{d_{i-1}^q}, \quad g_i = \frac{\ell_i}{\ell_{i-1}}, \quad i \geq 1,
\]
and uses them to find formulas for reciprocal sums.  We will interpolate such reciprocal sum formulas in \S\ref{S:sums}, but we see right away that
\begin{equation} \label{figi}
  f_i = f \bigl(\Xi^{(i)} \bigr), \quad g_i = \frac{\delta^{(i)} f^{(i)}}{\delta^{(i+1)}} \bigg|_{\Xi}, \quad i \geq 1.
\end{equation}
Thakur~\cite[Thm.~V]{Thakur92} finds exact formulas for $f_i$, $g_i$, using different methods, but by comparison
using~\eqref{fmotive} we find that our formula for $f_i$ agrees with Thakur's.  Using the dual $\bA$-motive below we recover Thakur's formula for $g_i$ as well.

We now discuss the dual $\bA$-motive $N$ associated to $\rho$, which leads to additional identities for $\rho$ involving the shtuka function.  The notion of dual $t$-motives is due to Anderson and was introduced in \cite[\S 4]{ABP04}.  Further properties of dual $\bA$-motives, including Anderson's constructions for connecting them with Drinfeld modules and Anderson $\bA$-modules, can be found in \cite[\S 4--5]{HartlJuschka16}.  As before, we let $L/K$ be an algebraically closed field, and we let
\begin{equation} \label{dualt}
  N = \Gamma\bigl(U, \cO_E(-(V^{(1)}))\bigr) \subseteq L[t,y].
\end{equation}
{From} the Riemann-Roch theorem, we have an $L$-basis for $N$,
\begin{equation}\label{dualbasis}
N = \Span_L \bigl( \delta^{(1)}, \delta f, \delta^{(-1)}f f^{(-1)}, \delta^{(-2)} f f^{(-1)} f^{(-2)}, \ldots \bigr).
\end{equation}
If we let $\sigma = \tau^{-1}$, then we can define a left $L[t,y,\sigma]$-module structure on $N$ by setting
\[
  \sigma h = f h^{(-1)}.
\]
With this action $N$ is a dual $\bA$-motive in the sense of Anderson~\cite[\S 4]{HartlJuschka16}, and we note that $N$ is an ideal of $L[t,y]$ and that it is a free left $L[\sigma]$-module of rank~$1$ generated by $\delta^{(1)}$.

As in the situation of the $\bA$-motive, we can use $N$ to construct a Drinfeld module, though in a more direct way (see \cite[\S 5.2]{HartlJuschka16}).  We define an $\F_q$-linear homomorphism $\varepsilon : N \to L$ by setting
\begin{equation}\label{varepseq}
  \varepsilon \bigl( c_0 \delta^{(1)} + c_1 \delta f + \cdots + c_i \delta^{(-i+1)} f f^{(-1)} \cdots f^{(-i+1)} \bigr)
  = c_0 + c_1^q + \cdots + c_i^{q^i}, \quad c_j \in L.
\end{equation}

\begin{lemma}\label{varepslemma}
The map $\varepsilon : N \to L$ is surjective and
\[
\ker(\varepsilon) = (1 - \sigma)N  =  \bigl\{ g \in N \mid g = h^{(1)} - fh\ \textup{for some\ } h
\in \Gamma(U,\cO_{E}(-(V))) \bigr\}.
\]
Thus $\varepsilon$ induces an isomorphism of $\F_q$-vector spaces, $\varepsilon : N/(1-\sigma)N \iso L$.
\end{lemma}

\begin{proof}
For $h \in \Gamma(U,\cO_E(-(V)))$, we have $h^{(1)} \in N$ and $\sigma(h^{(1)}) = fh$, so the two objects on the right are the same.  Also,
\[
  h^{(1)} = \sum_{i=0}^\ell c_i \delta^{(-i+1)} f f^{(-1)} \cdots f^{(-i+1)}
  \quad \Rightarrow\quad
  fh = \sum_{i=1}^{\ell+1} c_{i-1}^{1/q} \delta^{(-i+1)} f f^{(-1)} \cdots f^{(-i+1)},
\]
and it follows that $\varepsilon(h^{(1)}) = \varepsilon(fh)$.  Thus $(1-\sigma)N \subseteq \ker(\varepsilon)$.  To show that $\ker(\varepsilon) \subseteq (1-\sigma)N$, one shows that, for $g \in \ker(\varepsilon)$, the proposed relation $g = h^{(1)} - fh$ induces relations on the coefficients of $h$ that uniquely determine it.  We omit the details.
\end{proof}

Since elements of $\bA$ commute with $\sigma$, there is an induced $\bA$-module structure on $N/(1-\sigma)N$ and on $L$.  For $a \in \bA$, $\deg a = i$, if we write
\[
  a(t,y) \delta^{(1)} = a(\theta,\eta) \delta^{(1)} + c_1 \delta f  + c_2 \delta^{(-1)} f f^{(-1)}
  + \dots + c_i \delta^{(-i+1)} f f^{(-1)} \cdots f^{(-i+1)},
\]
then for any $x \in L$,
\[
  \varepsilon (a x \delta^{(1)}) = a(\theta,\eta)x + c_1^q x^q + \dots + c_i^{q^i} x^{q^i}.
\]
Thus if we define $\rho' : \bA \to L[\tau]$ by
\[
  \rho'_a = a(\theta,\eta) + c_1^q \tau + \cdots + c_i^{q^i} \tau^i,
\]
then one verifies that $\rho'$ is a Drinfeld module of rank~$1$ \cite[\S 5.2]{HartlJuschka16}, and moreover the map
\[
  \varepsilon : \frac{N}{(1 - \sigma)N} \iso \rho'(L)
\]
is an isomorphism of left $\bA$-modules, thus exemplifying the connection of Anderson between dual $\bA$-motives and Anderson $\bA$-modules (see \cite[\S 5]{HartlJuschka16}).  Again the action of $\rho'$ is determined by $\rho'_t$ and $\rho'_y$, and so we can find $z_1$, $w_1$, $w_2 \in L$ so that
\begin{gather}\label{teqdual}
t \delta^{(1)} = \theta \delta^{(1)} + z_1 \delta f + \delta^{(-1)} f f^{(-1)}, \\
\label{yeqdual}
y \delta^{(1)} = \eta \delta^{(1)} + w_1\delta f + w_2\delta^{(-1)} f f^{(-1)} + \delta^{(-2)} f f^{(-1)} f^{(-2)},
\end{gather}
and thus
\begin{equation}
\rho'_t = \theta + z_1^q \tau + \tau^2, \quad
\rho'_y = \eta + w_1^q \tau + w_2^{q^2}\tau^2 + \tau^3.
\end{equation}
The question of how $\rho$ and $\rho'$ are related is natural.

\begin{proposition}
The Drinfeld modules $\rho$ and $\rho'$ are equal.
\end{proposition}

\begin{proof}
This amounts to showing that $z_1^q = x_1$, $w_1^q = y_1$, and $w_2^{q^2} = y_2$, and by Theorem~\ref{T:DummitHayes}, it suffices to check that $z_1^q = x_1$.  Twisting \eqref{teqdual} and multiplying~\eqref{taction} through by $\delta$, we obtain
\[
  \delta^{(2)} t = \theta^q \delta^{(2)} + z_1^q \delta^{(1)} f^{(1)} + \delta f f^{(1)}, \quad
  \delta t = \theta \delta + x_1 \delta f  + \delta f f^{(1)}.
\]
Equating the $\delta f f^{(1)}$ terms from these equations, we obtain the equality
\[
  0 = \delta (t-\theta) - \delta^{(2)}(t-\theta^q) - x_1 \nu + z_1^q \nu^{(1)}.
\]
Recalling that $\delta = t - \alpha$, we see that the term $\delta(t-\theta) - \delta^{(2)}(t-\theta^q)$ has a pole of order at most $2$ at $\infty$.  On the other hand, $\nu$ and $\nu^{(1)}$ have poles of order $3$ at $\infty$.  Since $\tsgn(\nu) = \tsgn(\nu^{(1)}) = 1$, in order for the poles to cancel we must have $x_1 = z_1^q$.
\end{proof}

\begin{corollary}
For any $a \in \bA$ and $h \in N$,
\[
  \varepsilon(ah) = \rho_a(\varepsilon(h)).
\]
\end{corollary}

It turns out that different information about $\rho$ is contained in both the $\bA$-motive $M$ and the dual $\bA$-motive $N$.  For example, arguing as in the proof of Proposition~\ref{P:mx1}, we can use \eqref{teqdual}--\eqref{yeqdual} to show that for $i \geq 0$,
\begin{equation} \label{fdualmotive}
  \frac{\delta^{(i)} f^{(i)}}{\delta^{(i+1)}} = \frac{y - \eta^{q^i} - (y_2-x_1)^{q^{i-2}} (t-\theta^{q^i})}{t - \theta^{q^{i-1}} + y_1^{q^{i-1}} - x_1^{q^{i-1}}(y_2-x_1)^{q^{i-2}}},
\end{equation}
which is a companion formula to~\eqref{fmotive}.
Using \eqref{fdualmotive} to calculate $g_i = (\delta^{(i)} f^{(i)})/ \delta^{(i+1)} \big|_{\Xi}$ from \eqref{figi} and comparing the result with Thakur's formula for $g_i$~\cite[Thm.~V]{Thakur92}, we find that they are the same.  As companion formulas to~\eqref{mt0motive}, we find  by taking $i=2$ in \eqref{fdualmotive}, that
\begin{equation}
  m^{q^2} = y_2 - x_1, \quad \alpha^{q^3} = \theta^q - y_1^q + x_1^q(y_2-x_1).
\end{equation}

\section{The function $\omega_\rho$ and the period $\pi_\rho$}\label{S:OmegaRho}

In this section and the next we take on the considerations of Anderson, Sinha, and Thakur, regarding the function $\omega_C$ of \eqref{omegaC}, the Carlitz period $\tpi$ of~\eqref{Carlitzper}, and Anderson generating functions, in the context of our Drinfeld $\bA$-module~$\rho$ from~\S\ref{S:DrinfeldModules}.  Although the constructions of Sinha would likely succeed in this case, because our situation is completely concrete we can take a more direct route that relies less on homological algebra and rigid analysis.  Sinha's constructions use $t$-motives, but for us it is somewhat more convenient to use the dual $\bA$-motive $N = \Gamma(U,\cO_E(-(V^{(1)})))$ of $\rho$ from \eqref{dualt}.  The notation of previous sections is maintained throughout, with $L = \C_\infty$.

Recalling the Tate algebra $\TT_\theta$ from \eqref{Tatealgs}, its extension $\TT_\theta[y]$, and the associated rigid analytic space $\cU$ of points on the curve $E$, we let
\[
  \BB := \Gamma \bigl( \cU, \cO_E(-(V) + (\Xi))\bigr),
\]
which is a $\C_\infty[t,y]$-module of rigid analytic functions on $\cU$ that vanish at $V$ and have at most a simple pole at $\Xi$ (see~\cite[Chs.~3--4]{FresnelvdPut} for precise definitions).  We then define a space of functions on $\cU$ that satisfy certain difference properties,
\[
  \Omega := \bigl\{ h \in \BB \bigm| h^{(1)} - fh \in N \bigr\},
\]
which is naturally an $\bA$-module.  Of particular importance is the subspace
\[
  \Omega_0 := \bigl\{ h \in \Omega \bigm| h^{(1)} - fh = 0\bigr\}.
\]
We view functions $h \in \Omega$ as being mapped to rational functions in $N$ under the difference operator $\tau - f \in H(t,y)[\tau]$, which is similar to operators appearing elsewhere (e.g., see \cite[\S 3]{ABP04}, \cite[\S 4]{P08}, \cite[\S 2]{Pellarin12}, \cite[\S 4.2]{Sinha97}).  We define two further difference operators in $H(t,y)[\tau]$,
\begin{align}\label{Dt}
  D_t &= \rho_t - t = \theta - t + x_1\tau + \tau^2, \\
\label{Dy}
  D_y &= \rho_y - y = \eta - y + y_1\tau + y_2\tau^2 + \tau^3.
\end{align}
We note that for $h \in \Omega_0$, by~\eqref{taction} we have
\[
  D_t(h) = (\theta - t)h + x_1h\twist + h\twistk{2} = (\theta - t)h + x_1h f + h f f\twist = 0,
\]
and we similarly find from~\eqref{yaction} that $D_y(h) = 0$.  If a function vanishes under $D_t$ and $D_y$, then it vanishes under $\tau - f$ as well, as we see from the following proposition.

\begin{proposition}\label{P:tau-f}
As an element of $H(t,y)[\tau]$, the operator $\tau - f$ can be decomposed as
\begin{equation}\label{taufdecomp}
  \tau-f = \frac{1}{t-\alpha}\bigl(D_y - (\tau + m)D_t\bigr).
\end{equation}
\end{proposition}

\begin{proof}
We observe that the operators $D_t$ and $D_y$ factor as
\begin{align}
\label{Dtfactor}
D_t &= \rho_t - t = \biggl(\tau + \frac{t-\theta}{f}\biggr)(\tau - f ), \\
\label{Dyfactor}
D_y &= \rho_y - y = \biggl(\tau^2 + (f\twistk{2} + y_2 )\tau + \frac{y-\eta}{f}\biggr) (\tau-f ),
\end{align}
so that $D_t$, $D_y$ have $\tau-f$ as a right divisor.  We compute that
\[
  \tau^2 + (f\twistk{2} + y_2)\tau + \frac{y-\eta}{f}
  = (\tau + y_2-x_1^q ) \biggl(\tau + \frac{t-\theta}{f} \biggr) + t-\alpha,
\]
and using~\eqref{my2x1q}, \eqref{Dtfactor}, and \eqref{Dyfactor}, the decomposition for $\tau - f$ follows.  In essence we have shown that $\tau - f$ is a right greatest common divisor of $D_t$ and $D_y$ in $H(t,y)[\tau]$.
\end{proof}

In order to determine $\Omega_0$ exactly we define $\omega_\rho \in \Omega_0$, similar to $\omega_C$ in~\eqref{omegaC}.  We begin by fixing $(q-1)$-st roots of $-\alpha$ and $m \theta - \eta$, and then we set
\begin{align*}
  \delta_\rho &:= (-\alpha)^{1/(q-1)} \prod_{i=0}^\infty \biggl( 1 - \frac{t}{\alpha^{q^i}}
  \biggr)^{-1}, \\
  \nu_\rho &:= (m\theta - \eta)^{1/(q-1)} \prod_{i=0}^\infty \Biggl( 1 - \biggl( \frac{m}{m\theta - \eta} \biggr)^{q^i} t + \biggl( \frac{1}{m\theta -\eta} \biggr)^{q^i} y \Biggr)^{-1}.
\end{align*}
Since $\deg(\alpha) = 2$, it follows that the product for $\delta_{\rho}$ converges in $\TT$ with respect to the Gauss norm $\lVert\,\cdot\,\rVert$ and that $\delta_{\rho} \in \TT^{\times}$.  Moreover $\delta_{\rho} \in \Gamma(\cU,\cO_E((V)+(-V)))$.  By~\eqref{mquotients} $\deg(m) = q$, and so it similarly follows that $\nu_{\rho} \in \TT[y]^{\times}$, and furthermore $\nu_\rho \in \Gamma(\cU,\cO_E((-V)+(\Xi)))$.  (The constructions of $\delta_{\rho}$, $\nu_{\rho}$ should be compared with \cite[Eq.~(20)]{AnglesPellarinRibeiro16}.)  Applying twists,
\[
  \delta_{\rho}^{(1)} = \delta \cdot \delta_{\rho}, \quad \nu_{\rho}^{(1)} = \nu \cdot \nu_{\rho},
\]
and thus setting
\begin{equation} \label{omegarhodef}
  \omega_\rho := \frac{\nu_{\rho}}{\delta_{\rho}} \in \TT[y]^{\times},
\end{equation}
we have that $\omega_\rho \in \BB = \Gamma(\cU,\cO_E(-(V)+(\Xi)))$ and that
\begin{equation} \label{omegarhoprod}
  \omega_\rho = \xi^{1/(q-1)} \prod_{i=0}^\infty \frac{\xi^{q^i}}{f^{(i)}}, \quad \xi = -\frac{m\theta -\eta}{\alpha} = -\biggl( m + \frac{\beta}{\alpha} \biggr).
\end{equation}
Moreover by \eqref{fdef},
\begin{equation} \label{omegarhotwist}
  \omega_\rho^{(1)} = f \cdot \omega_\rho.
\end{equation}
We have thus proved the following lemma.

\begin{lemma} \label{L:omegarhoOmega0}
With notation as above, $\omega_\rho$ is an element of $\Omega_0$.
\end{lemma}

For other approaches to defining Anderson-Thakur functions for rank one Drinfeld modules on curves of arbitrary genera, the reader is directed to~\cite{AnglesNgoDacRibeiro16a}.

It is also apparent that $\omega_\rho$ extends meromorphically to all of $U$, with simple poles at $\Xi^{(i)}$, $i \geq 0$.  The following proposition is fundamental.

\begin{proposition} \label{Omegaprop}
The function $\omega_\rho$ generates $\Omega_0$ as a free $\bA$-module.
\end{proposition}

\begin{proof}
This is an adaptation of a result of Anderson and Thakur \cite[Lemma 2.5.4]{AndThak90} and its proof.  As noted above, $\omega_\rho \in \TT[y]^{\times}$, since all of its zeros and poles lie outside the inverse image under $t$ of the closed unit disk in $\C_\infty$.  Now for any $h \in \Omega_0$, we can let $g = h/\omega_\rho$, and it follows that $g^{(1)} = g$.  Therefore, $g \in \bA$, which is the fixed ring of $\TT[y]$ under twisting.  Thus $h \in \bA \omega_\rho$, and we are done.
\end{proof}

The function $\omega_\rho$ is an example of an Anderson generating function, which we will see in the course of the proof of Theorem~\ref{T:Resmap} below, and indeed the situation is similar to that of $\omega_C$ (see \cite[Prop.~2.2.5]{AndThak90}, \cite[\S 4]{EP14}).  We review here briefly some necessary facts about Anderson generating functions, but we note that they will be investigated more fully in \S\ref{S:AndGenFun}.  For $u \in \C_\infty$ we set
\begin{equation} \label{Eugeo}
  E_u(t) := \sum_{n=0}^\infty \exp_\rho \biggl( \frac{u}{\theta^{n+1}} \biggr) t^n \in \TT,
\end{equation}
and viewing $\rho$ as a rank $2$ Drinfeld $\F_q[t]$-module, we see from~\cite[\S 4.2]{Pellarin08} (or see \cite[Prop.~3.2]{EP14}) that $E_u$ extends meromorphically to all of $\C_\infty$ with simple poles at $t = \theta^{q^i}$, $i \geq 0$, via the partial fraction decomposition,
\begin{equation} \label{Euparfrac}
  E_u = \sum_{i=0}^\infty \frac{u^{q^i}}{d_i(\theta^{q^i} - t)}.
\end{equation}
In particular we have
\begin{equation} \label{ResEu}
  \Res_{t=\theta}(E_u\,dt) = -u, \quad \Res_{t = \theta^{q^i}} (E_u\,dt) = -\frac{u^{q^i}}{d_i}.
\end{equation}
Furthermore, we recall from \cite[\S 4.2]{Pellarin08} (or see \cite[Prop.~6.2]{EP14}) that
\begin{equation} \label{DtEu}
  D_t(E_u) = \exp_\rho(u).
\end{equation}
We define the Anderson generating function for $u$ associated to $\rho$ to be the function
\begin{equation} \label{Gudef}
  G_u := E_{\eta u} + (y+ a_1 t + a_3) E_{u} \in \TT[y],
\end{equation}
which extends meromorphically to all of $U$.  It has simple poles at $\Xi^{(i)}$, $i \geq 0$, with residues
\begin{equation} \label{GuResXitwist}
  \Res_{\Xi^{(i)}} (G_u \lambda) = \frac{-(\eta u)^{q^i} + (\eta^{q^i} + a_1 \theta^{q^i} + a_3) \bigl(-u^{q^i}\bigr)}{(2\eta^{q^i} + a_1 \theta^{q^i} + a_3) d_i} = -\frac{u^{q^i}}{d_i},
\end{equation}
and in particular
\begin{equation} \label{GuXiRes}
\Res_{\Xi} (G_u \lambda) = -u.
\end{equation}
One calculates that $\Res_{-\Xi^{(i)}}(G_u \lambda) = 0$ for $i \geq 0$, and so $G_u$ is regular at $-\Xi^{(i)}$ and moreover its only poles on $U$ are at $\Xi^{(i)}$, $i \geq 0$.  Now using \eqref{DtEu} we see that
\begin{equation} \label{DtGu}
  D_t(G_u) = \exp_{\rho}(\eta u) + (y + a_1 t + a_3)\exp_{\rho}(u),
\end{equation}
and after some calculation we find that
\begin{equation} \label{DyGu}
  D_y(G_u) = -a_1 \exp_\rho(\eta u) + \exp_\rho(\theta^2 u) + (t+a_2)\exp_\rho(\theta u) + (t^2 + a_2 t + a_4) \exp_\rho(u).
\end{equation}
Notably we see that $D_t(G_u)$, $D_y(G_u) \in \C_\infty[t,y]$.  Finally, we calculate $G_u(V)$, and we see from~\eqref{Euparfrac} that
\begin{equation} \label{GuV}
\begin{split}
  G_u(V) &= \sum_{i=0}^{\infty} \frac{\eta^{q^i} u^{q^i} + (\beta + a_1 \alpha + a_3) u^{q^i}}{d_i} \cdot \frac{1}{\theta^{q^i}-\alpha} \\
  &= \sum_{i=0}^{\infty} \frac{u^{q^i}}{d_i} \biggl( \frac{\eta^{q^i} + (\beta + a_1 \alpha + a_3)}{\theta^{q^i} - \alpha} \biggr) \\
  &= \sum_{i=0}^{\infty} \frac{u^{q^i}}{d_i} \bigl( f \bigl( \Xi^{(i)}\bigr) + m \bigr) \\
  &= \exp_{\rho}(u)^q + m \exp_{\rho}(u),
\end{split}
\end{equation}
where the third equality follows from~\eqref{fdef} and the last from~\eqref{diformula}.

Now fix $\pi \in \Lambda_\rho = \ker(\exp_\rho)$, and consider the function $G_{\pi}$.  We see from \eqref{DtGu} and \eqref{DyGu} that $D_t(G_{\pi}) = D_y(G_{\pi}) = 0$,
and thus by \eqref{taufdecomp} we have
\begin{equation}
  (\tau - f)(G_\pi) = G_{\pi}^{(1)} - f G_{\pi} = 0.
\end{equation}
Furthermore, from~\eqref{GuV} we see that $G_{\pi}(V) =0$.  Combining these calculations with \eqref{GuXiRes} we have established the following proposition.

\begin{proposition} \label{P:GpiOmega0}
For $\pi \in \Lambda_\rho$, the function $G_{\pi}$ is an element of $\Omega_0$, and furthermore $\Res_{\Xi}(G_\pi \lambda) = -\pi$.
\end{proposition}

Proposition~\ref{Omegaprop} implies that each non-zero $h \in \Omega_0$ has simple poles at each $\Xi^{(i)}$, $i \geq 0$.  Therefore taking residues at $\Xi$ defines a map $\Res: \Omega_0 \to \C_\infty$, where
\begin{equation} \label{Resdef}
  \Res(h) := \Res_\Xi(h\lambda), \quad h \in \Omega_0.
\end{equation}

\begin{theorem} \label{T:Resmap}
The map $\Res: \Omega_0\to \C_\infty$ is injective, and its image is $\Lambda_\rho = \ker(\exp_\rho)$.
\end{theorem}

\begin{proof}
The proof follows an argument of Anderson and Thakur~\cite[Prop. 2.2.5]{AndThak90} with a few modifications.  Let $h\in \Omega_0$.  As $h \in \TT[y]$, we deduce that we can express $h$ uniquely as
\[
  h = \sum_{n=0}^\infty b_{n+1}t^n + (y+a_1t + a_3)\sum_{n=0}^\infty c_{n+1}t^n.
\]
Since $\Omega_0 = \bA \omega_\rho$, we have $D_t(h) = 0$, and so $\rho_t(h) = th$ and
\begin{multline*}
\sum_{n=0}^\infty \rho_t(b_{n+1})t^n + (y+a_1t + a_3)\sum_{n=0}^\infty \rho_t(c_{n+1})t^n = \rho_t (h) \\
 = t  h = \sum_{n=0}^\infty b_{n+1}t^{n+1} + (y+a_1t + a_3)  \sum_{n=0}^\infty c_{n+1}t^{n+1}.
\end{multline*}
If we set $b_0=c_0=0$, then this calculation implies that for $n \geq 0$,
\begin{equation}\label{coefeq12}
\rho_t(b_{n+1}) = b_n, \quad \rho_t(c_{n+1}) = c_n.
\end{equation}
Similarly, since $D_y(h) = (\rho_y - y)(h) = 0$, we obtain further identities for $n \geq 0$,
\[
\rho_y(c_{n+1}) = b_{n+1}.
\]
Since $|b_n|$, $|c_n|\to 0$ as $n\to \infty$, there is some $n_0>0$ such that $b_{n+1}$ and $c_{n+1}$ both lie within the radius of convergence of $\log_\rho$ for $n>n_0$.  Thus by \eqref{coefeq12}, for $n>n_0$ we have
\[
  \theta^n \log_\rho(b_n) = \theta^{n+1}\log_\rho(b_{n+1}), \quad \theta^n \log_\rho(c_n) = \theta^{n+1}\log_\rho(c_{n+1}),
\]
and just as in the proof of \cite[Prop.~2.2.5]{AndThak90} we note that the two quantities are independent of $n$.  We set
\[
  \pi := \theta^n\log_\rho(c_n), \quad \textup{any}\ n > n_0,
\]
and note that
\[
\eta\pi = \eta \theta^n\log_\rho(c_n) = \theta^n\log_\rho(\rho_y(c_n)) = \theta^n\log_\rho(b_n).
\]
Since $\pi$ is independent of $n > n_0$, we see that
\[
\exp_\rho(\pi) = \exp_\rho(\theta^n\log_\rho(c_n)) = \rho_{t^n}(c_n) = \rho_t(c_1) =c_0= 0,
\]
which implies that $\pi \in \Lambda_\rho$.  Our calculations imply that
\[
b_n = \exp_\rho\biggl(\frac{\eta\pi}{\theta^n}\biggr), \quad c_n = \exp_\rho\biggl( \frac{\pi}{\theta^n}\biggr),
\]
and thus
\[
h = G_{\pi} = E_{\eta \pi} + (y+a_1t + a_3)E_{\pi}.
\]
By Proposition~\ref{P:GpiOmega0}, we see that $\Res(h) = -\pi$, and thus $\Res(\Omega_0) \subseteq \Lambda_\rho$.  Furthermore, we have shown that $\Omega_0$ is exactly the set of all Anderson generating functions $G_{\pi}$, $\pi \in \Lambda_{\rho}$, and so Proposition~\ref{P:GpiOmega0} implies that $\Res(\Omega_0) = \Lambda_\rho$.  Since $G_{\pi} = G_{\pi'}$ if and only if $\pi = \pi'$, it is also injective.
\end{proof}

{From} the preceding developments on $\omega_\rho$, if we let
\begin{equation} \label{pirhodef}
  \pi_\rho := -\!\Res(\omega_\rho) = -\!\Res_{\Xi}(\omega_\rho \lambda),
\end{equation}
then $\Lambda_{\rho} = A\pi_\rho$ and $\omega_\rho = G_{\pi_{\rho}}$.  Furthermore we obtain a product expansion for $\pi_\rho$, which provides a presentation of $\pi_\rho$ which is distinct from earlier ones of Gekeler~\cite[\S III]{Gekeler} (see \cite[\S 7.10]{Goss}, \cite[Ex.~4.15]{LP13}, and see Thakur~\cite[\S 3]{Thakur91} for further discussion and connections with $\Gamma$-functions). This formula should also be compared with the one for $\tpi$ in~\eqref{Carlitzper}.

\begin{theorem} \label{T:pirho}
We have $\Lambda_{\rho} = A \pi_\rho$, and setting $\xi = -(m + \beta/\alpha)$,
\[
  \pi_\rho = - \frac{\xi^{q/(q-1)}}{\delta^{(1)}(\Xi)} \prod_{i=1}^\infty \frac{\xi^{q^i}}{f^{(i)}(\Xi)}.
\]
\end{theorem}

\begin{proof}
We substitute into \eqref{omegarhoprod} and then apply \eqref{x1ftwist} to obtain
\[
\pi_\rho = \frac{-(t-\theta)}{2y + a_1t + a_3}\cdot \omega_\rho \bigg|_{\Xi}
 = -\xi^{q/(q-1)} \cdot \frac{x_1 + f^{(1)}(\Xi)}{2\eta + a_1\theta + a_3}\cdot
  \prod_{i=1}^{\infty} \frac{\xi^{q^i}}{f^{(i)}(\Xi)}.
\]
Using \eqref{alphaidentity} we arrive at the desired formula.
\end{proof}

\begin{remark}
The formula for $\pi_\rho$ can be made completely explicit, and after some calculation,
\[
  \pi_{\rho} = - \frac{\xi^{q/(q-1)}}{\theta - \alpha^q} \prod_{i=1}^{\infty}
  \left(\frac{1 - \dfrac{\theta}{\alpha^{q^i} \mathstrut }}{1 - \biggl( \dfrac{m}{m\theta - \eta}\biggr)^{q^{i \mathstrut}} \cdot \theta + \biggl( \dfrac{1}{m\theta - \eta} \biggr)^{q^{i\mathstrut}}\cdot \eta}\right).
\]
Similar to the Carlitz period, by observing that $\sgn(\xi) = -\sgn(m) = -1$ from \eqref{sgnalphabeta} and \eqref{mquotients}, we conclude that $\pi_\rho \in K_{\infty} \cdot \xi^{1/(q-1)}$ and moreover $\pi_\rho^j \in K_\infty$ if and only if $(q-1) \mid j$.
\end{remark}

\section{Applications of Anderson generating functions}\label{S:AndGenFun}

The theory of Anderson generating functions was originated by Anderson~\cite[\S 3.2]{And86} in his characterization of the uniformizability of $t$-modules in terms of rigid analytic trivializations.  As introduced in~\cite{AndThak90} the Anderson-Thakur function $\omega_C$ of \eqref{omegaC} provides a fundamental example of the utility of these generating series.   Subsequently they have been central in the study of periods, quasi-periods, $L$-series, and motivic Galois groups of Drinfeld modules and $t$-modules (e.g., see \cite{CP11}, \cite{CP12}, \cite{EP14}, \cite{Pellarin08}--\cite{Perkins14a}, \cite{Sinha97}).

In the present section we investigate applications of Anderson generating functions for our Drinfeld-Hayes module $\rho$ of \S\ref{S:DrinfeldModules}.  Inspired by Anderson and Thakur~\cite{AndThak90}, Sinha~\cite[\S 4]{Sinha97} used Anderson generating functions to give a formula for the exponential functions of certain rank~$1$ Anderson $t$-modules in terms of residues, and one of the goals of this section is to provide a version of Sinha's ``main diagram'' for $\rho$~\cite[\S 4.2.3, \S 4.6.6]{Sinha97}.  The Anderson generating functions we define provide explicit solutions of constructions of Anderson for $\exp_\rho(u)$ via dual $\bA$-motives (see \cite[\S 5.2]{HartlJuschka16}).  Although we do not expressly need Theorem~\ref{T:varepstheorem} in later sections, we feel it completes the picture for the exponential function in terms of difference equations started by $\omega_\rho$ and is of independent interest.

\begin{theorem}\label{T:varepstheorem}
Let $h \in \Omega = \{ h \in \BB \mid h^{(1)} - fh \in N \}$, where $N = \Gamma (U, \cO_E(-(V^{(1)}))) \subseteq \C_\infty[t,y]$, and denote
\[
u = -\!\Res_\Xi(h\lambda).
\]
Let $g = (\tau - f)(h) \in N$.  Then
\begin{equation}
\exp_\rho(u) = \varepsilon(g),
\end{equation}
where $\varepsilon$ is given in \eqref{varepseq}.
\end{theorem}

The Anderson generating functions $E_u$ and $G_u$ of \eqref{Eugeo} and \eqref{Gudef} are the main tools for proving this theorem.  We recall from~\eqref{DtGu} and~\eqref{DyGu} that
\[
  D_t(G_u),\ D_y(G_u) \in \C_\infty[t,y],
\]
and thus from \eqref{taufdecomp}
\[
  (\tau - f)(G_u) \in \frac{1}{t-\alpha} \cdot \C_\infty[t,y].
\]
Now $(\tau-f)(G_u)$ is nearly a polynomial, and this leads us to expect that we can use $G_u$ to manufacture a function in $\Omega$.  By~\eqref{GuV}, $G_u(V)$ need not be $0$, so we define
\begin{equation}\label{Judef}
  J_u := G_u - G_u(V).
\end{equation}
We note that $J_u(V) = 0$, and so $J_u \in \BB = \Gamma(\cU, \cO_E(-(V) + (\Xi)))$.

\begin{proposition}\label{P:taufJu}
Let $z = \exp_\rho(u)$.  For $J_u \in \BB$ as defined above, we have
\[
(\tau - f)(J_u) = (t-\alpha^q) z,
\]
and thus $J_u \in \Omega$.
\end{proposition}

\begin{proof}
From~\eqref{GuV} we have $G_u(V) = z^q + mz$, from which we calculate that
\begin{equation}\label{taufGuV}
(\tau - f)(G_u(V)) = z^{q^2} + (m^q - f)z^q - fm z.
\end{equation}
Using \eqref{taufdecomp}, \eqref{DtGu}, and \eqref{DyGu}, we find after a fairly straightforward computation that
\begin{equation}\label{allcoef}
\begin{split}
(\tau - f)(G_u) &= \frac{1}{t-\alpha} \biggl[ \begin{aligned}[t]
  \bigl(\theta^2 + \theta t + t^2  -m((y + a_1t + a_3) &{}+\eta) + a_4 \\
   &{}+ a_2(\theta + t) - a_1\eta\bigr) z
\end{aligned}
\\
&\quad {} + \bigl(x_1(\theta + \theta^q + t) - \eta^q  - y_1m - (y + a_1t + a_3) + a_2x_1 - a_1y_1 \bigr) z^q \\
&\quad {} + \bigl(x_1^{q+1} + \theta + \theta ^{q^2} + t - y_1^q - y_2m + a_2 - a_1y_2 \bigr) z^{q^2} \\
&\quad {} + \bigl(x_1 + x_1^q + x_1^{q^2} - y_2 - y_2^q - a_1 \bigr) z^{q^3} \biggr].
\end{split}
\end{equation}
Note that \eqref{allcoef} is a rational function in $\C_\infty(t,y)$.  The highest degree term in the numerator is $z t^2$ and the denominator is $t-\alpha = \delta$.  Thus $(\tau - f)(G_u)$ has degree $2$ and $\tsgn((\tau - f)(G_u)) = z$.  Its only possible poles away from $\infty$ occur at the zeros of $\delta$, namely $\pm V$.  We also observe from \eqref{taufGuV} that $(\tau - f)(G_u(V))\in \C_\infty(t,y)$ has degree $1$ and only a single possible simple pole away from $\infty$ at $V$.  Thus $(\tau - f)(J_u) \in \C_\infty(t,y)$ has degree $2$ with $\tsgn((\tau - f)(J_u)) = z$, and it has at most simple poles at $\pm V$.

Recall from the discussion following \eqref{Gudef} that $G_u$ extends to a meromorphic function on all of $U$ with simple poles only at $\Xi\twisti$ for $i\geq 0$, and from \eqref{Judef} that $J_u$ vanishes at $V$.  Using \eqref{divf}, these facts imply that $(\tau - f)(J_u)$ is regular at the points $\pm V$ and vanishes at $V\twist$.  Thus it is a rational function of degree 2 which is regular away from $\infty$ and hence is a degree $2$ element of $\C_\infty[t,y]$.  The Riemann-Roch theorem then implies that $(\tau - f)(J_u)$ is a $\C_\infty$-multiple of $t-\alpha^q$, and the fact that $\tsgn((\tau - f)(J_u)) = z$ finishes the proof.
\end{proof}

\begin{proof}[Proof of Theorem~\ref{T:varepstheorem}.]
Let $h\in \Omega$ and let $g = h\twist - fh\in N$.  As in Proposition~\ref{P:taufJu} we set $z = \exp_\rho(u)$.  We write $g$ in terms of the basis for the dual $A$-motive $N$ from \eqref{dualbasis},
\begin{equation}\label{gequation}
g = b_0 \delta\twist + b_1\twistk{-1} \delta f + b_2\twistk{-2} \delta\twistk{-1} ff\twistk{-1} + \dots + b_r\twistk{-r} \delta\twistk{-r+1} ff\twistk{-1} \cdots f\twistk{-r+1},
\end{equation}
where $b_i\in \C_\infty$ and $r = \deg g$.  We note that we have chosen to write the coefficients of \eqref{gequation} ``pre-twisted'' because then by \eqref{varepseq},
\begin{equation} \label{varepsg}
  \varepsilon(g) = b_0 + b_1 + \cdots + b_r.
\end{equation}
Now the right-hand side of the formula in Proposition~\ref{P:taufJu}, namely $(t-\alpha^q) z = z \delta^{(1)}$, is a constant multiple of the first basis element of \eqref{dualbasis}.  Note that $f^{(-1)} J_u^{(-1)} \in \BB$ and that
\[
(\tau-f) \bigl(f\invtwist J_u\invtwist \bigr) =  (f \tau\inv)(\tau-f)(J_u)= (f \tau\inv) \bigl(\delta\twist z\bigr) = \exp_\rho(u)^{1/q} \cdot \delta f,
\]
where $\delta f$ is the second basis element in \eqref{dualbasis}.  In like manner,
$f^{(-i)} f^{(-i+1)} \cdots f^{(-1)} J_u^{(-i)} \in \BB$ and
\begin{equation}
(\tau-f) \bigl(f\twistk{-i}f\twistk{-i+1}\cdots f\invtwist J_u\twistk{-i}\bigr) = \exp_\rho(u)^{1/q^i} \cdot \delta\twistk{-i+1} f f\invtwist \cdots f\twistk{-i+1},
\end{equation}
which is a multiple of the $(i+1)$-st basis element in \eqref{dualbasis}.
Now pick any $u_i \in \C_\infty$ so that $\exp_\rho(u_i) = b_i$, and define the function
\begin{equation}\label{gequation2}
J := J_{u_0} + f\invtwist J_{u_1}\invtwist + \dots + f\twistk{-r}f\twistk{-r+1}\cdots f\invtwist J_{u_r}\twistk{-r} \in \BB.
\end{equation}
Then by Proposition~\ref{P:taufJu} and the preceding developments,
\[
(\tau - f)(h - J) = 0,
\]
and so $h-J \in \Omega_0$.  Thus, by Proposition~\ref{Omegaprop}, there is some $a \in A$ so that
\[
h - J = \oa \omega_\rho = a(t,y) \omega_\rho.
\]
We also calculate the residue of an arbitrary term of $J$ at $\Xi$ by observing
\begin{align*}
\Bigl( \Res_\Xi \bigl(f\twistk{-i}f\twistk{-i+1} \cdots f\invtwist J_{u_i}\twistk{-i}\lambda \bigr) \Bigr)^{q^i}
&= \Res_{\Xi\twisti } (f f\twistk{1}\cdots f\twistk{i-1} J_{u_i}\lambda )\\
&= \bigl(f f\twistk{1}\cdots f\twistk{i-1}\big|_{\Xi\twisti} \bigr) \cdot  \Res_{\Xi\twisti }(J_{u_i}\lambda ) \\
&= d_i \cdot \Res_{\Xi^{(i)}} (J_{u_i} \lambda),
\end{align*}
where in the last equality we have used~\eqref{diformula}.  Using \eqref{GuResXitwist} we see that
\[
\Res_\Xi \bigl(f\twistk{-i}f\twistk{-i+1}\cdots f\invtwist J_{u_i}\twistk{-i}\lambda\bigr) = -u_i,
\]
and so by the above calculations and \eqref{pirhodef},
\[
u = -\!\Res_\Xi(h\lambda) = -\!\Res_\Xi ((J+a\omega_\rho) \lambda) = u_0 + u_1 + \dots +u_r + a {\pi}_\rho.
\]
Therefore, $\exp_\rho(u) = \exp_\rho(u_0 + u_1 + \dots + u_r) = b_0 + b_1 + \dots + b_r = \varepsilon(g)$ by \eqref{varepsg}.
\end{proof}

\section{Deformation and interpolation of reciprocal sums} \label{S:sums}

A classic result of Carlitz~\cite[Eq.~(9.09)]{Carlitz35} is a formula for the sum of reciprocals of monic polynomials of fixed degree,
\begin{equation} \label{Carlitzsum}
  \sum_{\substack{a \in \F_q[\theta]_+ \\ \deg a = i}} \frac{1}{a} = \frac{1}{(\theta - \theta^q)(\theta - \theta^{q^2}) \cdots (\theta - \theta^{q^i})}.
\end{equation}
These sums are intertwined with values of the Carlitz zeta function and Carlitz logarithms.  Pellarin~\cite{Pellarin12} discovered deformation formulas for these sums in the process of proving results on his now eponymous $L$-functions.  Further deformations and other related sums have been studied in \cite{AnglesPellarin14}, \cite{AnglesPellarin15}, \cite{AnglesSimon}, \cite{Perkins14a}, \cite{Perkins14b}, among other sources.

In~\cite{Thakur92}, Thakur initiated the study of reciprocal sums over more general rings $A$ and deduced formulas by pursuing the methods of Carlitz in the context of Drinfeld-Hayes modules.  However, Thakur discovered that there were some significant differences from the $\F_q[\theta]$ case and that the relationship with Drinfeld logarithms was especially subtle.  Anderson~\cite{And94} shed additional light on these matters in proving his first log-algebraicity results (see~\S\ref{S:LogAlg}).

In this section we investigate deformations and interpolations of reciprocal sums over $A = \F_q[\theta,\eta]$ and also over prime ideals of $A$, as precursors to investigations of Pellarin $L$-series in \S\ref{S:LSeries}.  Throughout we fix a non-zero prime ideal $\fp \subseteq A$, to which there is an associated point $P = (t_0,y_0) \in E(\F_q)$ such that $\fp = (\theta-t_0, \eta-y_0)$.  We let $\fp_+$ denote its monic elements and similarly $\fp_{i+} = \fp \cap A_{i+}$.  We define reciprocal sums
\begin{equation}
S_i  = \sum_{a\in A_{i+}} \frac{1}{a}, \quad S_{\fp,i} = \sum_{a\in \fp_{i+}} \frac{1}{a},
\end{equation}
and the main result of this section (Theorem~\ref{T:Siformula}) provides explicit interpolation formulas for $S_i$ and $S_{\fp,i}$ in terms of the shtuka function $f$.

Before proving our main results we have the following lemma, the first part of which is due to Thakur. We observe first that as in~\eqref{AFqbasis}, we can find an $\F_q$-basis for $\fp$,
\begin{equation} \label{pFqbasis}
  \fp = \Span_{\F_q} \bigl( (\theta - t_0)^j(\eta-y_0)^k \mid j \geq 0,\ k \in \{ 0,1 \},\ (j,k) \neq (0,0) \bigr),
\end{equation}
from which it follows that
\[
  \fp_{0+} = \fp_{1+} = \emptyset, \quad \fp_{2+} = \{ \theta - t_0 \}, \quad
  \fp_{3+} = \{ \eta - y_0 + c(\theta-t_0) \mid c \in \F_q \},
\]
and so on.  In particular, $S_{\fp,0} = S_{\fp,1} = 0$ and $S_{\fp,2} = 1/(\theta - t_0)$.

\begin{lemma} \label{L:SiDiprod}
For $i \geq 0$, let $D_i = \prod_{a \in A_{i+}} a$, and let $D_{\fp,i} = \prod_{a \in \fp_{i+}} a$.
\begin{enumerate}
\item[(a)] \textup{(Thakur~\cite[\S I]{Thakur92})} For $i \geq 2$,
\[
  S_i = (-1)^{i-1} \frac{(D_0 D_2D_3 \cdots D_{i-2}D_{i-1})^{q-1}}{D_i}.
\]
\item[(b)] For $i \geq 2$,
\[
  S_{\fp,i} = (-1)^{i} \frac{(D_{\fp,2}D_{\fp,3} \cdots D_{\fp,i-1})^{q-1}}{D_{\fp,i}},
\]
where the in the case $i=2$ the product in the numerator is empty.
\end{enumerate}
\end{lemma}

\begin{proof}
Observe that the lemma implies that for $i \geq 2$, $S_i \neq 0$ and $S_{\fp,i} \neq 0$.  The proofs of (a) and (b) are similar, and part (a) can be found in~\cite[Eqs.~(15), (19)]{Thakur92}.  We include the proof of (b) for completeness.  Checking the case of $S_{\fp,2} = 1/(\theta-t_0)$ is immediate.  For $i \geq 3$, let $z$ be an independent variable, and define the polynomial
\[
  e_i(z) = \prod_{\substack{a \in \fp \\ \deg a < i}} (z - a) \in A[z],
\]
and as the zero set of $e_i$ is an $\F_q$-vector space it follows that $e_i(z)$ is an $\F_q$-linear polynomial, $e_i(z) = B_0 z + B_1z^q + \cdots$.  We see that
\begin{equation} \label{B0prod}
  B_0 = \prod_{\substack{a \in \fp \\ \deg a < i, \, a \neq 0}} (-a)
  = \prod_{j=0}^{i-1}\: \prod_{\substack{a \in \fp_{j+}\\ c \in \F_q^{\times}}} ca
  = (-1)^{i-2} \bigl( D_{\fp,2} \cdots D_{\fp,i-1} \bigr)^{q-1}.
\end{equation}
Using \eqref{pFqbasis}, we choose $\kappa_i \in \fp_{i+}$ of degree $i$, and certainly $e_i(\kappa_i) = D_{\fp,i}$.  This implies that $e_i(z-\kappa_i) = e_i(z) - D_{\fp,i}$, and so
\[
  1 - \frac{e_i(z)}{D_{\fp,i}} = -\frac{e_i(z-\kappa_i)}{D_{\fp,i}} = -\frac{1}{D_{\fp,i}} \prod_{a \in \fp_{i+}} (z - a).
\]
If we take logarithmic derivatives with respect to $z$ and compare constant terms, we see that $B_0/D_{\fp,i} = S_{\fp,i}$, and the result follows from~\eqref{B0prod}.
\end{proof}

Thakur~\cite[Thm.~IV]{Thakur92} obtained another formula for $S_i$ in terms of $\ell_i$, $f_i$, and $g_i$ (see \eqref{liformula} and \eqref{figi}), but we determine new formulas by investigating deformations of $S_i$ and $S_{\fp,i}$, inspired by work on $\F_q[\theta]$ of Angl\`{e}s, Pellarin, and Simon~\cite{AnglesPellarin14}, \cite{AnglesSimon}.
We define functions on $E$ in $K[t,y]$ by
\[
\cS_i(t,y) = \sum_{a\in A_{i+}} \frac{\chi(a)}{a} = \sum_{a\in A_{i+}} \frac{a(t,y)}{a(\theta,\eta)},
\quad
\cS_{\fp,i}(t,y) = \sum_{a\in \fp_{i+}} \frac{\chi(a)}{a} = \sum_{a\in \fp_{i+}} \frac{a(t,y)}{a(\theta,\eta)}.
\]
In Proposition~\ref{P:SiSpi} we obtain product formulas for $\cS_i$ and $\cS_{\fp,i}$, and the main tool is a version of a lemma of Simon for $\bA$ (see \cite[Lem.~4]{AnglesPellarin14}).  Let $t_1, \dots, t_s$ be variables independent from $t$, and choose variables $y_j$ so that $t_j$ and $y_j$ satisfy the defining equation of $E$ in~\eqref{Eeq}.

\begin{lemma} \label{L:SimonslemmaforA}
Let $s \geq 1$ and $i \geq 2$.
\begin{enumerate}
\item[(a)] Define
\[
\cT_{i,s}(t_1,y_1, \dots, t_s,y_s) = \sum_{a\in A_{i+}} a(t_1,y_1)a(t_2,y_2)\cdots a(t_s,y_s) \in \F_q[t_1, y_1, \dots, t_s, y_s].
\]
Then $\cT_{i,s} = 0$ if and only if $s< (i-1)(q-1)$.
\item[(b)] Define
\[
\cT_{i,s}'(t_1,y_1, \dots, t_s,y_s) = \sum_{a\in \fp_{i+}} a(t_1,y_1)a(t_2,y_2)\cdots a(t_s,y_s) \in \F_q[t_1, y_1, \dots, t_s, y_s].
\]
Then $\cT_{i,s}' = 0$ if and only if $s < (i-2)(q-1)$.
\end{enumerate}
\end{lemma}

\begin{proof}
(cf.~\cite[Lem.~4]{AnglesPellarin14}) For $j = 0$ and $j \geq 2$, let $\kappa_j$ be the unique monomial in $t$ and $y$ from~\eqref{AFqbasis} of degree $j$, so that
\[
  \cT_{i,s} = \sum_{c_0, c_2, \dots, c_{i-1} \in \F_q} \prod_{r=1}^s \bigl(\kappa_i(t_r,y_r) + c_{i-1} \kappa_{i-1}(t_r,y_r) + \cdots + c_2 \kappa_{2}(t_r,y_r) + c_0 \bigr).
\]
Each term in this expanded product has the form
\[
  C_{j_1, \dots, j_s} \cdot \kappa_{j_1}(t_1,y_1) \cdots \kappa_{j_s}(t_s,y_s), \quad 0 \leq j_r \leq i, \ j_r \neq 1,
\]
where
\[
  C_{j_1, \dots, j_s} = \sum_{c_0, c_2, \dots, c_{i-1} \in \F_q} c_0^{\lambda_0} c_2^{\lambda_2} \cdots c_{i-1}^{\lambda_{i-1}}
\]
and $\lambda_j = \#\{ r \mid j_r = j \}$.  Now $\sum_j \lambda_j = s - \#\{ r \mid j_r = i \} \leq s$.  If $s < (i-1)(q-1)$, then for any $j_1, \dots, j_s$, at least one of $\lambda_0$, $\lambda_2, \dots, \lambda_{i-1}$ must be $< q-1$, from which it follows that the corresponding coefficient $C_{j_1, \dots, j_s}$ must be $0$.  Thus $\cT_{i,s} = 0$. On the other hand, if $s \geq (i-1)(q-1)$, then let $j_1 = \cdots = j_{q-1} = 0$; $j_{q} = \cdots = j_{2q-2} = 2$; and so on $j_{(i-2)(q-1)+1} = \cdots = j_{(i-1)(q-1)} = i-1$; and let $j_{(i-1)(q-1)+1} = \cdots = j_s = i$.  In this case $\lambda_0 = \lambda_2 = \cdots = \lambda_{i-1} = q-1$, and so $C_{j_1, \dots, j_s} \neq 0$, providing a non-zero term in $\cT_{i,s}$.  This proves (a).

For part (b), when $i=2$, we see that $\cT_{2,s} = (t_1-t_0) \cdots (t_s-t_0)$, which is never zero.  When $i \geq 3$, the rest of the proof is similar to (a), but instead we now set $\kappa_j$, $j \geq 2$, to be the unique monomial of degree $j$ in \eqref{pFqbasis}.  Then
\[
  \cT_{i,s}' = \sum_{c_2, \dots, c_{i-1} \in \F_q} \prod_{r=1}^s \bigl(\kappa_i(t_r,y_r) + c_{i-1} \kappa_{i-1}(t_r,y_r) + \cdots + c_2 \kappa_{2}(t_r,y_r) \bigr).
\]
From here the proof is the same as for $\cT_{i,s}$, noting that the sum is now over $i-2$ terms instead of $i-1$.
\end{proof}

\begin{proposition}\label{P:Sidiv}
As functions on $E$ we have the following equalities of divisors.
\begin{enumerate}
\item[(a)] For $i \geq 2$, $\divisor(\cS_i) = (\Xi) +( \Xi\twist) + \dots +  (\Xi\twistk{i-2}) + (V\twistk{i-1}-V) - i(\infty)$.
\item[(b)] For $i \geq 3$, $\divisor(\cS_{\fp,i}) = (\Xi) + (\Xi\twist) + \dots + (\Xi\twistk{i-3}) + (V\twistk{i-2} - V - P) + (P) - i(\infty)$.
\end{enumerate}
\end{proposition}

\begin{proof}
Let $i \geq 2$.  By Lemma~\ref{L:SiDiprod}(a), we see that $\tsgn(\cS_i) = S_i$ and moreover that $\deg(\cS_i) = i$.  We observe that for $j\geq 0$,
\begin{align*}
\cS_i \bigl(\Xi\twistk j \bigr) = \sum_{a\in A_{i+}} a(\theta,\eta)^{q^j-1} &= \sum_{a\in A_{i+}} a^{(q-1)(q^{j-1}+q^{j-2}+ \dots + 1)} \\
&= \cT_{i,(j-1)(q-1)}(\theta^{q^{j-1}},\eta^{q^{j-1}},\dots,\theta^{q^{j-1}},\eta^{q^{j-1}}, \dots, \theta,\eta,\dots,\theta,\eta),
\end{align*}
where each pair $\theta^{q^{j-k}}$, $\eta^{q^{j-k}}$ in the last line occurs $q-1$ times.  By Lemma~\ref{L:SimonslemmaforA}(a), we see that the above sum vanishes for $0\leq j \leq i-2$.  Now $\cS_i$ is regular on all of $U$ and has a single pole of order $i$ at $\infty$.  Since the induced sum on $E$ of the divisor of $\cS_i$ must be trivial, it follows that the zeros at $\Xi^{(j)}$, $0 \leq j \leq i-2$, must be simple and that $\cS_i$ has another simple zero at $-\Xi - \Xi^{(1)} - \cdots - \Xi^{(i-2)} = V^{(i-1)} - V$ (see \eqref{Vdef}).

Similarly for $i \geq 3$, Lemma~\ref{L:SiDiprod}(b) implies that $\tsgn(\cS_{\fp,i}) = S_{\fp,i}$ and that $\deg(\cS_{\fp,i}) = i$.  Using Lemma~\ref{L:SimonslemmaforA}(b), we see that $\cS_{\fp,i}(\Xi^{(j)}) = 0$ for $0 \leq j \leq i-3$.  We also have $\cS_{\fp,i}(P) = 0$.  The rest follows easily as in the previous paragraph using~\eqref{Vdef}.
\end{proof}

We also have $\divisor(\cS_{\fp,2}) = (-P) + (P) - 2(\infty)$, which fits in with Proposition~\ref{P:Sidiv}(b).  Let $\{g_i\}_{i=2}^\infty$ be the sequence of linear functions with $\tsgn(g_i)=1$ and divisor
\begin{equation}\label{divgi}
\divisor(g_i) = (V^{(i-1)} - V) + (-V^{(i-1)}) + (V) - 3( \infty),
\end{equation}
and let $\{g_{\fp,i}\}_{i=2}^\infty$ be the sequence of functions with $\tsgn(g_{\fp,i}) = 1$ and divisor
\begin{equation}\label{divgfpi}
\divisor(g_{\fp,i}) = (V^{(i-2)} - V -P) + (-V^{(i-2)}) + (V) + (P) - 4(\infty).
\end{equation}
By~\eqref{fdef}, \eqref{nudiv}, and Proposition~\ref{P:Sidiv}, we obtain the the following proposition and its corollary by comparing divisors and leading terms.

\begin{proposition} \label{P:SiSpi}
For $i \geq 2$, we have the following equalities in $H(t,y)$.
\begin{enumerate}
\item[(a)] $\cS_i = S_i \cdot \dfrac{g_i}{\nu^{(i-1)}} \cdot f f^{(1)} \cdots f^{(i-1)} \vphantom{\bigg|}$.
\item[(b)] $\cS_{\fp,i} = S_{\fp,i} \cdot \dfrac{g_{\fp,i}}{\nu^{(i-2)}} \cdot f f^{(1)} \cdots f^{(i-2)} \vphantom{\bigg|}$.
\end{enumerate}
\end{proposition}

The function $g_i$ can be written explicitly as
\begin{equation} \label{giformula}
g_i = y - \beta + \frac{\beta^{q^{i-1}}+a_1\alpha^{q^{i-1}}+a_3+\beta}{\alpha^{q^{i-1}} - \alpha}\cdot (t-\alpha).
\end{equation}
We will find it more convenient to deal with $g_i$ after dividing by $\delta$, so we define
\begin{equation}\label{lambdai}
\lambda_i := \frac{g_i}{\delta} = \frac{y - \beta}{t-\alpha} + \frac{\beta^{q^{i-1}}+a_1\alpha^{q^{i-1}}+a_3+\beta}{\alpha^{q^{i-1}} - \alpha},
\end{equation}
and record that
\begin{equation}\label{divlambdai}
\divisor\left (\lambda_i\right ) = (-V\twistk{i-1}) - (-V) + (V\twistk{i-1} - V) - (\infty).
\end{equation}
To write $g_{\fp,i}$ explicitly we define functions $\lambda_{\fp,i}$ and $G_{\fp}$ to have $\tsgn = 1$ and divisors
\begin{gather}\label{divlambdafpi}
\divisor(\lambda_{\fp,i}) = (-V^{(i-2)}) - (-V-P) +(V^{(i-2)} - V-P)  - (\infty), \\
\label{divGfp}
\divisor(G_\fp) =  -(-V) + (P) + (-V-P) - (\infty).
\end{gather}
(The function $G_\fp$ should not be confused with the function $G_u$ in~\eqref{Gudef}.)  We find that
\begin{equation}\label{gdeltaquotient}
\frac{g_{\fp,i}}{\delta} = \lambda_{\fp, i}\cdot G_\fp.
\end{equation}
Just as we derived for $\lambda_i$ in~\eqref{lambdai}, we can write $\lambda_{\fp,i}$ as
\begin{equation}\label{lambdafpi}
\lambda_{\fp,i} = \frac{y - y(V+P)}{t-t(V+P)} + \frac{\beta^{q^{i-2}}+a_1 \alpha^{q^{i-2}} + a_3+y(V+P)}{\alpha^{q^{i-2}} - t(V+P)}
\end{equation}
(as part of verifying that the right-hand side has the correct divisor from~\eqref{divlambdafpi} we note that $V+P$ is not a $2$-torsion point since $V$ is transcendental over the field of definition $\F_q$ of $E$).
If we take $\nu_\fp(t,y)$ to be the linear polynomial connecting the collinear points $V$, $P$, $-(V+P)$ and denote its slope as $m_\fp$, then we can write $G_\fp$ as
\begin{equation}\label{Gkeq}
G_\fp = \frac{\nu_\fp(t,y)}{\delta(t)} = \frac{(y-y_0) - m_\fp(t-t_0)}{t-\alpha}.
\end{equation}
Using \eqref{gdeltaquotient}--\eqref{Gkeq}, we find $g_{\fp,i} = \lambda_{\fp,i} \cdot \nu_{\fp}$ explicitly.

The formulas for $\cS_i$ and $\cS_{\fp,i}$ will be important in~\S\ref{S:LSeries}--\ref{S:LogAlg}, but it remains to prove formulas for $S_i$ and $S_{\fp,i}$ themselves.  We derive such formulas as specializations of products of twists of shtuka functions, and our results complement a formula of Thakur~\cite[Thm.~IV]{Thakur92} for $S_i$.  (The advantage of these formulas over those in Lemma~\ref{L:SiDiprod} is that we do not have closed formulas for $D_i$ and $D_{\fp,i}$.)

\begin{theorem} \label{T:Siformula}
For $S_i = \sum_{a \in A_{i+}} 1/a$ and $S_{\fp,i} = \sum_{a \in \fp_{i+}} 1/a$, the following formulas hold.
\begin{enumerate}
\item[(a)] For $i\geq 2$,
\[
S_i = \frac{\nu^{(i)}}{g_i^{(1)}\cdot f^{(1)} \cdots f^{(i)}} \Bigg|_\Xi.
\]
\item[(b)] For $i\geq 2$,
\[
S_{\fp,i} = \frac{\nu^{(i-1)}}{g_{\fp,i}^{(1)} \cdot f^{(1)} \cdots f^{(i-1)}} \Bigg|_{\Xi}.
\]
\end{enumerate}
\end{theorem}

\begin{proof}
Again we adapt methods of Angl\`{e}s, Pellarin, and Simon (see \cite[\S 2]{AnglesPellarin14}, \cite{AnglesSimon}, \cite[Thm. 5.2.6]{PLogAlg}).  For $1 \leq s \leq q-1$, choose variables $t_1$, $y_1, \dots, t_s$, $y_s$ as in Lemma~\ref{L:SimonslemmaforA}, and define a multivariable deformation of $S_i$,
\[
  \cQ_{i,s}(t_1,y_1,\dots,t_s,y_s) = \sum_{a \in A_{i+}} \frac{a(t_1,y_1) \cdots a(t_s,y_s)}{a(\theta,\eta)}.
\]
For $j \geq 0$, evaluating the final $(t_s,y_s)$ variables of $\cQ_{i,s}$ at $\Xi^{(j)}$ gives
\begin{multline*}
\cQ_{i,s}(t_1,y_1,t_2,y_2,\dots,\theta^{q^j}, \eta^{q^j}) = \sum_{a \in A_{i+}} a(t_1,y_1) \cdots a(t_{s-1,} y_{s-1})a(\theta,\eta)^{q^j-1} \\
= \cT_{i,s-1+j(q-1)}(t_1,y_1,\dots,t_{s-1},y_{s-1},\theta^{q^{j-1}},\eta^{q^{j-1}},\dots,\theta^{q^{j-1}},\eta^{q^{j-1}}, \dots, \theta,\eta,\dots,\theta,\eta),
\end{multline*}
where, similar to the proof of Proposition~\ref{P:Sidiv}, the last line has $q-1$ entries of each pair $\theta^{q^{j-k}}$, $\eta^{q^{j-k}}$.  By Lemma~\ref{L:SimonslemmaforA}(a), as $s \leq q-1$ we find that the above quantity vanishes for $j\leq i-2$.  By symmetry, we find that $\cQ_{i,s}(t_1,y_1,\dots,t_s,y_s)$ vanishes at $\Xi^{(j)}$ for $j\leq i-2$ and for each of the~$s$ pairs of variables.  We claim that
\begin{equation} \label{Qisprod}
  \cQ_{i,s} = S_i\cdot \prod_{r=1}^s \frac{g_i}{\nu^{(i-1)}} \cdot f f^{(1)} \cdots f^{(i-1)} \bigg|_{(t_r,y_r)}.
\end{equation}
We consider $\cQ_{i,s}$ to be a function on the $s$-fold product $E \times \cdots \times E$.  Using the same analysis as in the proofs of Propositions~\ref{P:Sidiv} and~\ref{P:SiSpi}, we see that
\[
  G := \cQ_{i,s} \biggm/ \biggl( \frac{g_i}{\nu^{(i-1)}} \cdot f f^{(1)} \cdots f^{(i-1)} \bigg|_{(t_1,y_1)} \biggr) \in \oK(t_1,y_1, \dots, t_s, y_s)
\]
must be independent of $t_1$ and $y_1$.  That is, $G \in \oK(t_2,y_2, \dots, t_s,y_s)$.  Proceeding by induction it follows that
\[
\cQ_{i,s} \biggm/ \biggl( \prod_{r=1}^s \frac{g_i}{\nu^{(i-1)}} \cdot f f^{(1)} \cdots f^{(i-1)} \bigg|_{(t_r,y_r)} \biggr) \in \oK,
\]
and by comparing leading coefficients this quotient must be $S_i$, thus verifying~\eqref{Qisprod}.  Now setting $t_1 = \cdots = t_s = t$ and $y_1 = \cdots = y_s = y$, we define
\[
\cR_{i,s}(t,y) = \cQ_{i,s}(t,y,\dots,t,y) = \sum_{a \in A_{i+}} \frac{a(t,y)^s}{a(\theta,\eta)} \in K[t,y],
\]
and using \eqref{Qisprod} we see that
\[
  \cR_{i,s}(t,y) = S_i \biggl( \frac{g_i}{\nu^{(i-1)}} \cdot f f^{(1)} \cdots f^{(i-1)} \biggr)^s.
\]
Taking $s = q-1$ and twisting, we find
\[
  \cR_{i,q-1}^{(1)}(t,y) = \sum_{a \in A_{i+}} \frac{a(t,y)^{q-1}}{a(\theta,\eta)^q} = S_i^q
  \biggl( \frac{g_i^{(1)}}{\nu^{(i)}} \cdot f^{(1)} \cdots f^{(i)}
  \biggr)^{q-1},
\]
and evaluating at $\Xi$, we have
\[
 S_i = S_i^q \biggl( \frac{g_i^{(1)}}{\nu^{(i)}} \cdot f^{(1)} \cdots f^{(i)}
  \biggr)^{q-1}\bigg|_{\Xi}.
\]
Solving for $S_i$, for some $\zeta \in \F_q^{\times}$, we have
\[
  S_i = \zeta\cdot \frac{\nu^{(i)}}{g_i^{(1)} \cdot f^{(1)} \cdots f^{(i)} } \bigg|_{\Xi}.
\]
The proof for part (a) will be complete once we show that $\zeta = 1$.  We analyze the signs of the terms on the right-hand side.  From~\eqref{fdef},
\[
  \nu^{(j)}(\Xi) = \eta - \eta^{q^j} - m^{q^j} \bigl(\theta - \theta^{q^j} \bigr), \quad
  \delta^{(j)}(\Xi) = \theta - \alpha^{q^j}.
\]
By~\eqref{mquotients} we see that
\[
  \sgn(m) = \frac{\sgn(\beta^q)}{\sgn(\alpha^q)} = 1,
\]
since $\sgn(\alpha) = \sgn(\beta) = 1$ by~\eqref{sgnalphabeta}.  For $j \geq 1$, the term in $\nu^{(j)}(\Xi)$ of highest degree is $m^{q^j}\theta^{q^j}$, and thus
\[
  \sgn\bigl( \nu^{(j)}(\Xi) \bigr) = \sgn \bigl( m^{q^j} \theta^{q^j} \bigr) = 1, \quad
  \sgn\bigl( \delta^{(j)}(\Xi) \bigr) = \sgn \bigl( -\alpha^{q^j} \bigr) = -1.
\]
Finally, using~\eqref{giformula} we see from a similar analysis that $\sgn (g_i^{(1)}(\Xi)) = -1$. Therefore,
\[
  \sgn \biggl( \frac{\nu^{(i)}}{g_i^{(1)} \cdot f^{(1)} \cdots f^{(i)} } \bigg|_{\Xi} \biggr) = (-1)^{i-1} = \sgn(S_i),
\]
where the last equality follows from Lemma~\ref{L:SiDiprod}(a).

For (b), we note first that $S_{\fp,2} = 1/(\theta-t_0)$, and since $g_{\fp,2} = (t-t_0)\delta$ by~\eqref{divgfpi}, the result for $S_{\fp,2}$ follows.  For $i \geq 3$, the rest of the proof of (b) is similar to (a).  We use Lemma~\ref{L:SimonslemmaforA}(b) to analyze an analogous function
\[
 \cQ_{i,s}'(t_1,y_1,\dots,t_s,y_s) = \sum_{a \in \fp_{i+}} \frac{a(t_1,y_1) \cdots a(t_s,y_s)}{a(\theta,\eta)},
\]
and through a comparable argument we arrive at the identity
\[
  S_{\fp,i} = \zeta \cdot \frac{\nu^{(i-1)}}{g_{\fp,i}^{(1)} \cdot f^{(1)} \cdots f^{(i-1)}} \Bigg|_{\Xi},
\]
for some $\zeta \in \F_q^{\times}$.  By Lemma~\ref{L:SiDiprod}(b), $\sgn(S_{\fp,i}) = (-1)^i$, and we show this agrees with the sign of the right-hand side when $\zeta =1$.  To do this we must also use~\eqref{gdeltaquotient}--\eqref{Gkeq} to observe that $\sgn(\lambda_{\fp,i}(\Xi)) = 1$ and that $\sgn(G\twist_{\fp}(\Xi)) = \sgn(\delta\twist(\Xi)) = -1$, from which we confirm that $\sgn(g_{\fp,i}^{(1)}(\Xi)) = -1$.  For the sake of brevity we leave the remaining details to the reader.
\end{proof}

\section{Pellarin $L$-series} \label{S:LSeries}

In this section we investigate two types of Pellarin $L$-series for $\bA$, and we prove versions of Pellarin's identity~\eqref{Pellarinformula} for them in terms of $\omega_\rho$ and $\pi_\rho$.  For $s \in \Z_+$ and $\chi$ as defined in~\eqref{caniso}, we set
\begin{equation}
  L(\bA;s) = \sum_{a \in A_+} \frac{\chi(a)}{a^s} = \sum_{a \in A_+} \frac{a(t,y)}{a^s},
\end{equation}
and it is easily seen that, for fixed $s$, $L(\bA;s)$ is an element of $\TT[y]$.  However, one can use the methods of \cite[\S 5.3]{AnglesNgoDacRibeiro16b} and~\cite{Goss13} to show that $L(\bA;s)$ extends to an entire function on $U$.  We have the following identity.

\begin{theorem} \label{T:Lvalue1}
As elements of $\TT[y]$,
\[
  L(\bA;1) = -\frac{\delta^{(1)}\, \pi_\rho}{f \omega_\rho}.
\]
\end{theorem}

\begin{remark}
Pellarin~\cite[Thm.~2]{Pellarin12} proved special value results for $L(\F_q[t];s)$ for $s \equiv 1 \pmod{q-1}$ along the lines of \eqref{Pellarinformula}, and these results were extended to multivariable versions of $L(\F_q[t];s)$ in~\cite{AnglesPellarin14}, \cite{AnglesPellarinRibeiro15}, \cite{AnglesPellarinRibeiro16}, \cite{PellarinPerkins16}, \cite{Perkins14a}.  For the sake of space we do not pursue analogous results for $L(\bA;s)$ here, though we expect such questions to lead to interesting investigations.
\end{remark}

To define the second type of Pellarin $L$-series, we recall the action of non-zero ideals $\fa \subseteq A$ on the isomorphism classes of the Drinfeld-Hayes module $\rho : \bA \to H[\tau]$, as defined by Hayes~\cite{Hayes79}.  For $\fa \subseteq A$,  we define the left ideal of $H[\tau]$ by
\begin{equation} \label{Jfa}
  J_{\fa} = \langle \rho_{\oa} \mid a \in \fa \rangle \subseteq H[\tau].
\end{equation}
As $H[\tau]$ is a left principal ideal domain~\cite[Cor.~1.6.3]{Goss}, there is a unique monic generator $\rho_\fa$ of $J_\fa$, and we have $\deg_\tau \rho_\fa = \deg \fa = \dim_{\F_q}(A/\fa)$.  Hayes proved~\cite[Prop.~3.2]{Hayes79} that for $\fa \neq 0$ there is a unique Drinfeld-Hayes module $\fa * \rho : \bA \to H[\tau]$ with the property for $\fb \subseteq A$,
\begin{equation} \label{rhofafb}
  \rho_{\fa \fb} = (\fa * \rho)_{\fb} \rho_{\fa}.
\end{equation}
We note that if $a \in A_+$, then $\rho_{(a)} = \rho_{\oa}$.  If we let $\sigma_{\fa} \in \Gal(H/K)$ denote the Artin automorphism associated to $\fa$, then Hayes proved~\cite[Thm.~8.5]{Hayes79} that
\[
  \fa * \rho = \rho^{\sigma_{\fa}},
\]
where $\rho^{\sigma_{\fa}} : \bA \to H[\tau]$ is the Drinfeld-Hayes module defined by letting $\sigma_\fa$ act on the coefficients of $\rho$.  Furthermore, if we let $B \subseteq H$ be the integral closure of $A$ in $H$, then $\rho_\fa \in B[\tau]$ and $B \fa = B \pd (\rho_{\fa})$, where $\pd(\rho_{\fa})$ is the constant term of $\rho_\fa$ with respect to $\tau$~\cite[Thm.~15.9]{Hayes92}.

For an ideal $\fa \subseteq A$, we set
\begin{equation} \label{chifadef}
  \chi(\fa) := \frac{\rho_\fa(\omega_\rho)}{\omega_\rho}.
\end{equation}
A priori we see that $\chi(\fa) \in \TT[y]$, but we will show in Lemma~\ref{L:chiprops} that $\chi(\fa) \in H(t,y)$.  We note right away, from the discussion after~\eqref{Dt}--\eqref{Dy} and from Lemma~\ref{L:omegarhoOmega0}, that
\begin{equation} \label{chiprincideal}
  \chi((a)) = \oa = a(t,y) = \chi(a), \quad a \in A_+,
\end{equation}
and so the definition of $\chi$ is consistent with the one on $A_+$ used earlier in this section.
For integers $s \geq 1$, we define the Pellarin $L$-series,
\begin{equation}
\LL(\bA;s) := \sum_{\fa \subseteq A} \frac{\chi(\fa)}{\pd(\rho_\fa)^s},
\end{equation}
where again we adopt the convention that such sums will always be over non-zero ideals $\fa \subseteq A$.  We will see that $\LL(\bA;s)$ is an element of $\TT[y]$.  Since $\pd(\rho_{(a)})= \pd(\rho_{\oa}) = a$, it follows that the subsum of $\LL(\bA;s)$ over principal ideals is exactly $L(\bA;s)$.  Our second main result is the following.

\begin{theorem} \label{T:Lambdavalue1}
As elements of $\TT[y]$,
\[
  \LL(\bA;1) = -\sum_{\sigma \in \Gal(H/K)} \biggl( \frac{\delta^{(1)}}{f} \biggr)^{\sigma} \cdot \frac{\pi_\rho}{\omega_\rho},
\]
where $H$ is the Hilbert class field of $K$.
\end{theorem}

In the interim, Angl\`{e}s, Ngo Dac, and Tavares Ribeiro~\cite{AnglesNgoDacRibeiro16a} have proved that a function similar to $\LL(\bA;1)$ is a rational multiple of $\pi_{\rho}/\omega_{\rho}$ using different methods.  However, although their methods work for arbitrary genera, they do not obtain a precise formula for this multiple.

We also prove the following class number identity in $\F_q$ upon evaluation at $\Xi$.

\begin{corollary} \label{C:Lambdaclassno}
Let $h(\bA) = \#E(\F_q)$ be the class number of $\bA$. Then
\[
\LL(\bA;1)\big|_{\Xi} = h(\bA).
\]
\end{corollary}

\begin{remark}
It would be interesting to make a comparison between $\LL(\bA;s)$ and similar Goss $L$-series obtained through exponentiation of ideals (see~\cite{Goss83}, \cite{Goss92}, \cite[Ch.~8]{Goss}).
\end{remark}

The proofs of Theorems~\ref{T:Lvalue1} and~\ref{T:Lambdavalue1} require a few preliminary results, but the primary underlying tool is the following proposition.  Recall from~\S\ref{S:sums} the definition of $\cS_i$ and, for a prime ideal $\fp \subseteq A$ of degree~$1$, the definition of $\cS_{\fp,i}$.

\begin{proposition} \label{P:partialsums}
We have the following partial sums in $H[t,y]$.
\begin{enumerate}
\item[(a)] For $i \geq 2$,
\[
\sum_{j=0}^i \cS_{j} = \frac{f^{(1)} \cdots f^{(i)}\cdot \dfrac{g_{i}^{(1)}}{\nu^{(i)}}}{f^{(1)} \cdots f^{(i)}\cdot  \dfrac{g_{i}^{(1)}}{\nu^{(i)}} \bigg|_{\Xi}}.
\]
\item[(b)] For a prime ideal $\fp \subseteq A$ of degree~$1$ and for $i \geq 2$,
\[
\sum_{j=0}^{i} \cS_{\fp, j} = \frac{f^{(1)} \cdots f^{(i-1)}\cdot \dfrac{g_{\fp,i}^{(1)}}{\nu^{(i-1)}} }{ f^{(1)} \cdots f^{(i-1)}\cdot \dfrac{g_{\fp,i}^{(1)}}{\nu^{(i-1)}} \bigg|_\Xi}.
\]
\end{enumerate}
\end{proposition}

\begin{proof}
Combining Proposition~\ref{P:SiSpi}(a) and Theorem~\ref{T:Siformula}(a) we have
\[
\cS_i = S_i \cdot f f^{(1)} \cdots f^{(i-1)}\cdot \frac{g_i}{\nu^{(i-1)}}
= \frac{f f^{(1)} \dots f^{(i-1)} \cdot \dfrac{g_i}{\nu^{(i-1)}}}{f^{(1)} f^{(2)}\cdots f^{(i)} \cdot \dfrac{g_i^{(1)}}{\nu^{(i)}} \bigg|_{\Xi}}.
\]
We then argue by induction on $i$.  We use $\cS_0 = 1$ and $\cS_1=0$ to establish the base case for $i=2$.  Examining the divisors given in \eqref{divf}, \eqref{nudiv}, and \eqref{divgi}, we claim that
\[
f\twist f\twistk{2} \frac{g_2\twist}{\nu\twistk{2}}  - \biggl(f\twist f\twistk{2} \frac{g_2\twist}{\nu\twistk{2}}\bigg|_{\Xi}\biggr)  = ff\twist\frac{g_2}{\nu\twist}.
\]
Indeed both sides vanish at $\Xi$ and have polar divisor $-2(\infty)$, which implies that the two sides share the same divisor.  Both sides have $\tsgn=1$, and so they must be equal to each other (and incidentally to $t-\theta$).  This implies that
\begin{equation}\label{basecase}
\cS_0 + \cS_1 + \cS_2 = 1+\frac{ff\twist\dfrac{g_2}{\nu\twist}}{f\twist f\twistk{2} \dfrac{g_2\twist}{\nu\twistk{2}} \bigg|_{\Xi}} = \frac{f\twist f\twistk{2} \dfrac{g_2\twist}{\nu\twistk{2}}}{f\twist f\twistk{2} \dfrac{g_2\twist}{\nu\twistk{2}} \bigg|_{\Xi}},
\end{equation}
and establishes case $i=2$.  For the induction step, examining divisors and signs shows that for $i \geq 3$,
\[
f\twist f\twistk{2}\cdots f\twistk{i} \frac{g_i\twist}{\nu\twistk{i}} - \biggl(f\twistk{i} \frac{g_i\twist}{\nu\twistk{i}}\cdot \frac{\nu\twistk{i-1}}{g_{i-1}\twist} \bigg|_{\Xi} \cdot f\twist f\twistk{2}\cdots f\twistk{i-1} \frac{g_{i-1}\twist}{\nu\twistk{i-1}}\biggr)  = ff\twist \cdots f\twistk{i-1} \frac{g_i}{\nu\twistk{i-1}}.
\]
Applying the induction hypothesis and then substituting in this identity, we find
\begin{align*}
\cS_0+ \cS_1 + \dots + \cS_{i-1} +  \cS_{i} &= \frac{f\twist f\twistk{2}\cdots f\twistk{i-1} \cdot \dfrac{g_{i-1}\twist}{\nu\twistk{i-1}}}{f\twist f\twistk{2}\cdots f\twistk{i-1} \cdot \dfrac{g_{i-1}\twist}{\nu\twistk{i-1}} \bigg|_{\Xi}}  +  \frac{ff\twist \cdots f\twistk{i-1} \cdot \dfrac{g_i}{\nu\twistk{i-1}}}{f\twist f\twistk{2}\cdots f\twistk{i} \cdot \dfrac{g_i\twist}{\nu\twistk{i}} \bigg|_{\Xi}} \\
&= \frac{f\twist f\twistk{2}\cdots f\twistk{i}\cdot \dfrac{g_i\twist}{\nu\twistk{i}}}{f\twist f\twistk{2}\cdots f\twistk{i} \cdot \dfrac{g_i\twist}{\nu\twistk{i}} \bigg|_{\Xi}},
\end{align*}
which establishes part (a).

For part (b) we know that $\cS_{\fp,0} = \cS_{\fp,1} = 0$ and that $\cS_{\fp,2} = (t-t_0)/(\theta - t_0)$.  Using that $g_{\fp,2} = (t-t_0)\delta$, the case $i=2$ follows from direct calculation.  To illustrate the induction step we consider the $i=3$ case.  Analyzing the divisors in~\eqref{divf}, \eqref{nudiv}, and~\eqref{divgfpi} as above, we claim that
\[
f^{(1)} f^{(2)} \cdot \frac{g_{\fp,3}^{(1)}}{\nu^{(2)}} - \frac{t - t_0}{\theta-t_0} \cdot \biggl( f^{(1)} f^{(2)} \cdot \frac{g_{\fp,3}^{(1)}}{\nu^{(2)}} \bigg|_{\Xi} \biggr)
= f f^{(1)} \cdot \frac{g_{\fp,3}}{\nu^{(1)}}.
\]
To prove the claim we see that the divisor of the right-hand side is $(\Xi) + (P) + (V^{(1)} - V-P) - 3(\infty)$, whereas the left-hand side has polar divisor $-3(\infty)$ and vanishes at both $\Xi$ and~$P$.  Since both sides have $\tsgn=1$, they are equal.  The case $i=3$ then follows.  The general induction is similar to this case and to part (a), and we omit the details.
\end{proof}

To prove Theorem~\ref{T:Lvalue1} we then combine Proposition~\ref{P:partialsums}(a) with the following proposition.

\begin{proposition} \label{P:quotientlimits}
The following limits hold in $\TT[y]$ with respect to the Gauss norm $\lVert \,\cdot\,\rVert$.
\begin{enumerate}
\item[(a)] $\displaystyle \lim_{i \to \infty} \frac{\nu^{(i)}}{\nu^{(i)}(\Xi)} = 1 \vphantom{\Bigg|}$.
\item[(b)] $\displaystyle \lim_{i \to \infty} \frac{g_i^{(1)}}{g_i^{(1)}(\Xi)} = \frac{\delta^{(1)}}{\delta^{(1)}(\Xi)}\vphantom{\Bigg|}$.
\item[(c)] $\displaystyle \lim_{i \to \infty} \frac{f^{(1)} \cdots f^{(i)}}{f^{(1)} \cdots f^{(i)}|_{\Xi}} = -\frac{\delta^{(1)}(\Xi) \cdot \pi_\rho}{f \omega_\rho}\vphantom{\Bigg|}$.
\end{enumerate}
\end{proposition}

\begin{proof}
For (a), we see from~\eqref{fdef} that
\[
  \frac{\nu^{(i)}}{\nu^{(i)}(\Xi)} = \frac{(m\theta - \eta)^{q^i} + y - m^{q^i}t}{(m\theta - \eta)^{q^i} + \eta - m^{q^i}\theta},
\]
and that $(m\theta - \eta)^{q^i}$, which has degree~$(q+2)q^i$, dominates both the numerator and denominator. Therefore, dividing through top and bottom by $(m\theta - \eta)^{q^i}$, we find that the limit goes to~$1$.  For (b), we note from~\eqref{giformula} that
\[
  \frac{g_i^{(1)}}{g_i^{(1)}(\Xi)} = \frac{(\alpha^{q^i}-\alpha^q)(y - \beta^q) + (\beta^{q^i} + a_1 \alpha^{q^i} + a_3 + \beta^q)(t-\alpha^q)}{(\alpha^{q^i}-\alpha^q)(\eta - \beta^q) + (\beta^{q^i} + a_1 \alpha^{q^i} + a_3 + \beta^q)(\theta-\alpha^q)}.
\]
The $\beta^{q^i} + a_1 \alpha^{q^i} + a_3 + \beta^q$ factors dominate the numerator and denominator, and so the limit goes to $(t-\alpha^q)/(\theta-\alpha^q) = \delta^{(1)}/\delta^{(1)}(\Xi)$ as desired.  As for (c), we recall the product formulas for $\omega_\rho$ and $\pi_\rho$ in~\eqref{omegarhoprod} and Theorem~\ref{T:pirho}, from which we see that
\[
  \lim_{i \to \infty} \frac{f^{(1)} \cdots f^{(i)}}{f^{(1)} \cdots f^{(i)}|_{\Xi}}
  =\lim_{i \to \infty} \frac{\prod_{j=1}^i f^{(j)}/\xi^{q^j}}{\prod_{j=1}^i f^{(j)}(\Xi)/\xi^{q^j}} = -\frac{\delta^{(1)}(\Xi)\cdot \pi_{\rho}}{f \omega_\rho},
\]
and we are done.
\end{proof}

\begin{proof}[Proof of Theorem~\ref{T:Lvalue1}]
By definition $L(\bA;1) = \sum_{j=0}^\infty \cS_j$, and so by Proposition~\ref{P:partialsums}(a) we see that
\[
  L(\bA;1) = \lim_{i\to \infty} \frac{\nu^{(i)}(\Xi)}{\nu^{(i)}} \cdot \frac{g_i^{(1)}}{g_i^{(1)}(\Xi)}
  \cdot \frac{f^{(1)} \cdots f^{(i)}}{f^{(1)} \cdots f^{(i)}|_{\Xi}} = -\frac{\delta^{(1)}\,\pi_\rho}{f \omega_\rho},
\]
where the last equality follows from Proposition~\ref{P:quotientlimits}.
\end{proof}

We now turn to the proof of Theorem~\ref{T:Lambdavalue1}.  Although it is similar to the proof of Theorem~\ref{T:Lvalue1}, we need to make the precise connections between $\LL(\bA;1)$ and the sums $\sum_{j=0}^\infty \cS_{\fp,j}$.  The proof is also somewhat complicated by the fact that $\lim_{i \to \infty} g_{\fp,i}^{(1)}/g_{\fp,i}^{(1)}(\Xi)$ is not as directly computable as in Proposition~\ref{P:quotientlimits}(b).  To overcome these obstacles we require several preliminary results, many of which fortunately will also be important for~\S\ref{S:LogAlg}.

We momentarily fix a prime ideal $\fp = (\theta-t_0,\eta-y_0) \subseteq A$ corresponding to a point $P = (t_0,y_0) \in E(\F_q)$.  We will assume that $P$ has order $r$ in $E(\F_q)$, and so the class of $\fp$ has order $r$ in the ideal class group $\Cl(A)$.  We let $R_{\fp} \subseteq H(t,y)$ be the local ring of functions regular at $P$.  We will let $\sigma = \sigma_{\fp} \in \Gal(H/K)$ denote the Artin automorphism associated to~$\fp$.

\begin{lemma} \label{L:rhofp}
We have $\rho_{\fp} = \tau-f(P)$.  In particular $\pd(\rho_\fp) = -f(P)$.
\end{lemma}

\begin{proof}
As noted above $\rho_{\fp}$ is the unique monic generator of the left ideal $J_{\fp}$ from~\eqref{Jfa}. Now~$J_{\fp}$ is generated by $\rho_{t-t_0}$ and $\rho_{y-y_0}$, and we observe that
\begin{align*}
\rho_{t-t_0} - (t-t_0) &= \rho_{t} - t = D_t, \\
\rho_{y-y_0} - (y-y_0) &= \rho_{y} - y = D_y,
\end{align*}
which are both right divisible by $\tau -f$ in $H(t,y)[\tau]$ by~\eqref{Dtfactor}--\eqref{Dyfactor}.  Consider the function $\varphi : R_{\fp}[\tau] \to H[\tau]$ defined by evaluating the coefficients of a polynomial in $R_{\fp}[\tau]$ at~$P$.  Since~$P$ is $\F_q$-rational, one checks that for any $g \in R_{\fp}$ and $i \geq 0$,
\[
  g^{(i)}(P) = g(P)^{q^i},
\]
and so $\varphi$ is an $\F_q$-algebra homomorphism.  From~\eqref{taufdecomp} we see that
\begin{align*}
\varphi(\tau -f) = \tau - f(P) &= \frac{1}{\delta(P)}\cdot \bigl( \rho_{y-y_0} - (y_0-y_0) - (\tau + m) (\rho_{t-t_0} - (t_0-t_0)) \bigr) \\
&= \frac{1}{\delta(P)}\cdot \bigl( \rho_{y-y_0}  - (\tau + m) \rho_{t-t_0} \bigr),
\end{align*}
and thus $\tau -f(P) \in J_{\fp}$.  Since it is monic of least possible positive degree in $J_{\fp}$ it must be a generator, and so $\rho_{\fp} = \tau - f(P)$.
\end{proof}

In order to simplify $\LL(\bA;1)$, we next outline properties of $\chi(\fa)$, which we recall from~\eqref{chifadef}.  One can also use this lemma to show that $\LL(\bA;s) \in \TT[y]$ for any $s \in \Z_+$.

\begin{lemma} \label{L:chiprops}
The function $\chi$ on non-zero ideals $\fa \subseteq A$ satisfies the following properties.
\begin{enumerate}
\item[(a)] For $a \in A_+$, $\chi((a)) = \oa$.
\item[(b)] If $c \in K$, with $\sgn(c) = 1$, and $\fa \subseteq A$ satisfy $c\fa \subseteq A$, then
\[
  \chi(c\fa) = \oc \chi(\fa).
\]
\item[(c)] For a degree~$1$ prime $\fp \subseteq A$ corresponding to a point $P \in E(\F_q)$,
\[
  \chi(\fp) = f - f(P).
\]
\end{enumerate}
In particular for all $\fa \subseteq A$, we have $\chi(\fa) \in H(t,y)$.
\end{lemma}

\begin{proof}
Part (a) was observed in~\eqref{chiprincideal}.  For (b), if $c \in A_+$, then by~\eqref{rhofafb}, $\rho_{c\fa} = \rho_{\fa} \rho_{\oc}$, and so
\[
 \rho_{c\fa}(\omega_\rho) = \rho_{\fa} \rho_{\oc}(\omega_\rho) = \rho_{\fa}(\oc \omega_{\rho}) = \oc \rho_{\fa}(\omega_{\rho}) = \oc \chi(\fa) \omega_{\rho},
\]
and so $\chi(c\fa) = \oc\chi(\fa)$.  On the other hand if $c = a/b$ with $a$, $b \in A_+$, then $a\fa = (bc)\fa = b(c\fa)$, and so by the previous argument we have $\oa\chi(\fa) = \ob \chi(c\fa)$ as desired.  For (c), Lemmas~\ref{L:omegarhoOmega0} and~\ref{L:rhofp} imply that
\[
  \rho_{\fp}(\omega_\rho) = f\omega_\rho - f(P)\omega_\rho \quad \Rightarrow \quad \chi(\fp) = f - f(P).
\]
For the final part we note that every non-zero ideal $\fa \subseteq A$ is either principal or equivalent to a prime ideal of degree~$1$, and so $\chi(\fa) \in H(t,y)$ by (a)--(c).
\end{proof}

For a fixed non-zero ideal $\fb \subseteq A$, we now introduce the sum
\[
  \tLambda_{\fb} = \sum_{\fa \sim \fb} \frac{\chi(\fa)}{\pd(\rho_{\fa})},
\]
where the sum is over all integral ideals equivalent to $\fb$.  We note that the function
\[
  c \mapsto c\fb : (\fb^{-1})_+ \to \{ \fa \subseteq A \mid \fa \sim \fb \}
\]
is a bijection.  Using~\eqref{rhofafb} and arguing similarly to Lemma~\ref{L:chiprops}(b), we have
\begin{equation} \label{pdrhocfb}
\pd(\rho_{c\fb}) = c \pd(\rho_{\fb}), \quad c \in \fb^{-1}.
\end{equation}
Taking these together with Lemma~\ref{L:chiprops}, we see that
\begin{equation} \label{tLambdafb}
\tLambda_{\fb} = \frac{\chi(\fb)}{\pd(\rho_{\fb})} \cdot \sum_{c \in (\fb^{-1})_+}
\frac{c(t,y)}{c(\theta,\eta)}.
\end{equation}
By definition, if $\fb' \sim \fb$, then $\tLambda_{\fb'} = \tLambda_{\fb}$, and if $\fb_1, \dots, \fb_h \subseteq A$ represent the ideal classes of $A$, then $\LL(\bA;1) = \tLambda_{\fb_1} + \cdots + \tLambda_{\fb_h}$.

On the other hand in order to work with sums over elements of ideals themselves, so as to appeal to Proposition~\ref{P:partialsums}(b), we change our point of view on $\tLambda_{\fb}$ slightly.  Suppose that $\fa \subseteq A$ is an integral ideal with $\fa \sim \fb^{-1}$.  Then there is a unique $\gamma \in A_+$ so that $\fa\fb = (\gamma)$ and thus the function
\[
  a \mapsto \frac{a}{\gamma} : \fa \to \fb^{-1},
\]
is a bijection.  By~\eqref{tLambdafb}
\[
  \tLambda_{\fb} = \frac{\chi(\fb)}{\pd(\rho_{\fb})} \cdot \sum_{a \in \fa_+} \frac{\oa/\ogamma}{a/\gamma} = \frac{\chi(\gamma \fa^{-1})/\ogamma}{\pd(\rho_{\gamma\fa^{-1}})/\gamma} \cdot \sum_{a \in \fa_+} \frac{a(t,y)}{a(\theta,\eta)}.
\]
This prompts a definition for any non-zero ideal $\fa \subseteq A$: pick any $\gamma \in A_+$ so that $\gamma \fa^{-1} \subseteq A$ and set
\begin{equation} \label{Lambdafa}
  \Lambda_{\fa} := \frac{\chi(\gamma \fa^{-1})/\ogamma}{\pd(\rho_{\gamma\fa^{-1}})/\gamma} \cdot
  \sum_{a \in \fa_+} \frac{a(t,y)}{a(\theta,\eta)}.
\end{equation}
In this way $\Lambda_{\fa} = \tLambda_{\gamma \fa^{-1}}$, and since $\tLambda_{\fb}$ depends only on the ideal class of $\fb$, we see that $\Lambda_{\fa}$ is independent of the choice of $\gamma$.  Thus if we take $\fp_2, \dots, \fp_h \subseteq A$ to be the prime ideals of degree~$1$, which represent the non-trivial ideal classes of $A$, then
\begin{equation} \label{Lambda1sum}
  \LL(\bA;1) = \Lambda_{(1)} + \Lambda_{\fp_2} + \cdots + \Lambda_{\fp_h}.
\end{equation}
This necessitates the study of the quantities $\pd(\rho_{\gamma \fp^{-1}})$ and $\chi(\gamma \fp^{-1})$ for our fixed prime ideal~$\fp$ and corresponding point~$P \in E(\F_q)$ of order $r$.

\begin{lemma} \label{L:gammaquotient}
Let $\gamma \in A_+$ be chosen uniquely so that $\fp^r = (\gamma)$.  For $\sigma = \sigma_\fp \in \Gal(H/K)$, we have
\[
  \frac{\gamma}{\pd(\rho_{\gamma\fp^{-1}})} = \frac{\gamma}{\pd(\rho_{\fp^{r-1}})} =   -f(P)^{\sigma^{-1}}.
\]
\end{lemma}

\begin{proof}
The first equality is immediate since $\gamma \fp^{-1} = \fp^{r-1}$.  From~\eqref{rhofafb},
\[
  \rho_{\ogamma} = \rho^{\sigma^{r-1}}_\fp \cdot \rho_{\fp^{r-1}},
\]
and so $\gamma = \pd(\rho_{\ogamma}) = \pd(\rho_{\fp}^{\sigma^{r-1}}) \cdot \pd(\rho_{\fp^{r-1}})$.  Thus we see that
\[
  \frac{\gamma}{\pd(\rho_{\fp^{r-1}})} = \pd(\rho_{\fp}^{\sigma^{r-1}}) =
  -f(P)^{\sigma^{r-1}} =  -f(P)^{\sigma^{-1}}
\]
by Lemma~\ref{L:rhofp}.
\end{proof}

\begin{lemma} \label{L:Vsigma}
For $P \in E(\F_q)$, with corresponding ideal $\fp \subseteq A$, and Galois automorphism $\sigma=\sigma_{\fp} \in\Gal(H/K)$, the following hold.
\begin{enumerate}
\item[(a)] $V^{\sigma} = V-P$.
\item[(b)] $f(P) = f(V^{\sigma})$ and $f(-P) = f(V^{\sigma^{-1}})$.
\item[(c)] $f(V^{\sigma}) f^{\sigma}(V) = f(P) f(-P)^{\sigma} = \theta-t_0$.
\end{enumerate}
\end{lemma}

\begin{proof}
We observe first that $V^\sigma$ also satisfies equation \eqref{Vdef}, namely $(1-\Fr)(V^\sigma) = \Xi$, and thus $V^\sigma$ is equal to $V+P'$ for some $P'\in E(\F_q)$.  Coordinate-wise we have the identity $\Xi \equiv P \pmod{\fp}$, and therefore from \eqref{Vdef} we have $V - P \equiv V^{(1)} \pmod{\fp}$.  It follows that
\[
t(V-P) \equiv \alpha^q \pmod{\fp}, \quad
y(V-P) \equiv \beta^q \pmod{\fp}.
\]
Since $H = K(\alpha,\beta)$, the definition of the Artin automorphism $\sigma_{\fp}$ implies that $t(V-P) = t(V^{\sigma})$ and $y(V-P) = y(V^{\sigma})$, confirming that $P' = -P$ and proving (a).

For (b) we see from \eqref{fdef} that
\[
  f(P) = -m+\frac{y_0+\beta+a_1\alpha + a_3}{t_0-\alpha},
\]
and we notice that the fraction on the right-hand side is the slope of the line connecting the three points $P$, $-V$, and $V-P$.  However, since $V-P = V^\sigma$ we observe that
\[
f(P) = -m + \frac{y(V^\sigma) + \beta+a_1\alpha + a_3}{t(V^\sigma) - \alpha} = f(V^\sigma).
\]
Arguing analogously we find that $f(-P) = f(V^{\sigma^{-1}})$.

Now, (b) implies $f(V^\sigma)f^\sigma(V) = f(P)f(-P)^\sigma$, which is the first part of (c).  To prove the other equality in (c), we let $\fq = (\theta-t_0,\eta+y_0+a_1 t_0 + a_3)$ be the ideal corresponding to $-P$.  Since $\fp\fq = (\theta-t_0)$, we apply~\eqref{rhofafb} to find
\[
  \rho_{t-t_0} = \rho_{\fp\fq} = \rho^{\sigma}_{\fq} \rho_{\fp},
\]
which implies $\theta - t_0 = \pd(\rho^{\sigma}_{\fq}) \pd(\rho_{\fp})$.  By Lemma~\ref{L:rhofp}, $\pd(\rho_\fp) = f(P)$ and $\pd(\rho^{\sigma}_{\fq}) = f(-P)^{\sigma}$.
\end{proof}

Consider the function $f - f(P)$.  We see that it has simple poles at $V$ and $\infty$ and no other poles, and it vanishes at $P$.  It follows that
\begin{equation} \label{divf-fP}
  \divisor(f-f(P)) = (P) + (V-P) - (V) - (\infty).
\end{equation}
Recall also the function $G_{\fp}$ from~\eqref{Gkeq}.  The following lemma gives a description of $G_{\fp}^{(1)}$ evaluated at $\Xi$.

\begin{lemma} \label{L:Gfptwist}
Let $P \in E(\F_q)$, with corresponding ideal $\fp \subseteq A$, and Galois automorphism $\sigma=\sigma_{\fp} \in\Gal(H/K)$.  For the function $G_{\fp}$ defined in \eqref{Gkeq},
\[
  G_{\fp}^{(1)}(\Xi) = - \frac{f(P) \delta(P)}{\delta^{(1)}(P)} = - f^{\sigma^{-1}}(P).
\]
\end{lemma}

\begin{proof}
We see that $\divisor(G_{\fp}^{(1)}) = (P) + (-V^{(1)}-P) - (-V^{(1)}) - (\infty)$.  Let
$\Gamma := G_{\fp}^{(1)} - G_{\fp}^{(1)}(\Xi)$.  Now just as in the analysis of~\eqref{divf-fP}, we deduce that
\[
  \divisor(\Gamma) = (-V) + (\Xi) - (-V^{(1)}) - (\infty) = \divisor \biggl(
  \frac{f\delta}{\delta^{(1)}} \biggr),
\]
where the last equality is checked directly.  Also $\tsgn(\Gamma) = 1 = \tsgn(f\delta/\delta^{(1)})$, and so
\[
  G_{\fp}^{(1)} - G_{\fp}^{(1)}(\Xi) = \frac{f\delta}{\delta^{(1)}}.
\]
Evaluating both sides at $P$ yields
\[
  G_{\fp}^{(1)}(\Xi) = -\frac{f(P) \delta(P)}{\delta^{(1)}(P)},
\]
which is the first equality.  Now by Lemma~\ref{L:Vsigma}(a), $\divisor(f^{\sigma^{-1}}) = (V^{(1)}+P) - (V+P) + (\Xi) - (\infty)$, and thus by~\eqref{divf-fP},
\[
  \divisor \biggl( \frac{f^{\sigma^{-1}} \cdot (f - f(-P))}{f^{(1)} - f^{(1)}(-P)} \biggr)
  = (V^{(1)}) - (V) + (\Xi) - (\infty) = \divisor(f).
\]
The functions have the same sign, so
\begin{equation} \label{fsigmaprod}
  \frac{f^{\sigma^{-1}} \cdot (f - f(-P))}{f^{(1)} - f^{(1)}(-P)} = f.
\end{equation}
A quick calculation from~\eqref{fdef} reveals that
\[
  f(P) - f(-P) = \frac{2y_0 + a_1t_0 + a_3}{\delta(P)},
\]
and moreovoer, $f^{(1)}(P) - f^{(1)}(-P) = (f(P) - f(-P))^{(1)} = (2y_0 + a_1t_0 + a_3)/\delta^{(1)}(P)$, since~$P$ is $\F_q$-rational.  Therefore, evaluating both sides of~\eqref{fsigmaprod} at $P$ yields the second equality in the statement of the lemma.
\end{proof}

\begin{proof}[Proof of Theorem~\ref{T:Lambdavalue1}]
If we take $\fa = (1)$ in \eqref{Lambdafa}, then we quickly see that $\Lambda_{(1)} = L(\bA;1)$.  By Theorem~\ref{T:Lvalue1} and~\eqref{Lambda1sum}, we are done if we can show that for a prime ideal $\fp = (\theta-t_0,\eta-y_0) \subseteq A$, with corresponding point $P = (t_0,y_0) \in E(\F_q)$ and Artin automorphism $\sigma = \sigma_{\fp} \in \Gal(H/K)$,
\begin{equation} \label{LambdafpRed}
  \Lambda_{\fp} = -\biggl( \frac{\delta^{(1)}}{f} \biggr)^{\sigma^{-1}} \cdot \frac{\pi_\rho}{\omega_\rho}.
\end{equation}
We calculate $\Lambda_\fp$ by starting with \eqref{Lambdafa}.  Choose $\gamma \in A_+$ so that $\fp^r = (\gamma)$, where $r$ is the order of $\fp$ in $\Cl(A)$.  Let $\fq = (\theta -t_0, \eta+y_0 + a_1t_0 + a_3)$ correspond to $-P$, so that $\fp\fq = (\theta-t_0)$ and therefore $(\theta-t_0) \gamma \fp^{-1} = \gamma \fq$.  By Lemma~\ref{L:chiprops},
\[
\chi(\gamma\fq) = \ogamma \chi(\fq)  = \ogamma(f - f(-P)),
\]
but on the other hand $\chi(\gamma \fq) = \chi((\theta-t_0) \gamma \fp^{-1}) = (t-t_0) \chi(\gamma \fp^{-1})$.  Combining these equations,
\begin{equation} \label{chigammafpinv}
\chi(\gamma\fp^{-1}) = \frac{\ogamma (f - f(-P))}{t-t_0}.
\end{equation}
Substituting this equation and the formula from Lemma~\ref{L:gammaquotient} into~\eqref{Lambdafa}, we have
\begin{equation} \label{Lambdafp1}
  \Lambda_{\fp} = -\frac{f^{\sigma^{-1}}(P) (f - f(-P))}{t-t_0} \sum_{a \in \fp_+}
  \frac{a(t,y)}{a(\theta,\eta)}.
\end{equation}
By Proposition~\ref{P:partialsums}(b), we see that
\begin{equation} \label{sumSpj1}
  \sum_{a \in \fp_+} \frac{a(t,y)}{a(\theta,\eta)} = \lim_{i \to \infty} \sum_{j=0}^{i+1} \cS_{\fp,j}
   = \lim_{i \to \infty} \frac{\nu^{(i)}(\Xi)}{\nu^{(i)}} \cdot \frac{g_{\fp,i+1}^{(1)}}{g_{\fp,i+1}^{(1)}(\Xi)}
  \cdot \frac{f^{(1)} \cdots f^{(i)}}{f^{(1)} \cdots f^{(i)}|_{\Xi}}.
\end{equation}
We recall from \eqref{gdeltaquotient} that $g_{\fp,i+1} = \delta \cdot \lambda_{\fp,i+1} \cdot G_{\fp}$, and in a similar manner to the proof of Proposition~\ref{P:quotientlimits}(b), we verify that in $\TT[y]$ we have $\lim_{i \to \infty} \lambda_{\fp,i+1}^{(1)}/\lambda_{\fp,i+1}^{(1)}(\Xi) = 1$,
and so
\[
  \lim_{i \to \infty} \frac{g_{\fp,i+1}^{(1)}}{g_{\fp,i+1}^{(1)}(\Xi)} =
  \frac{\delta^{(1)}}{\delta^{(1)}(\Xi)} \cdot \frac{G_{\fp}^{(1)}}{G_{\fp}^{(1)}(\Xi)}.
\]
Thus using this limit and Proposition~\ref{P:quotientlimits} in~\eqref{sumSpj1}, we find
\[
  \sum_{a \in \fp_+} \frac{a(t,y)}{a(\theta,\eta)} = - \frac{G_{\fp}^{(1)}}{G_{\fp}^{(1)}(\Xi)} \cdot \frac{\delta^{(1)} \cdot \pi_{\rho}}{f\omega_\rho},
\]
and substituting this into~\eqref{Lambdafp1} and using Lemma~\ref{L:Gfptwist}, we obtain
\begin{equation} \label{Lambdafp2}
  \Lambda_{\fp} = - \frac{(f-f(-P)) G_{\fp}^{(1)}}{t-t_0} \cdot \frac{\delta^{(1)} \cdot \pi_\rho}{f \omega_\rho}.
\end{equation}
Now by~\eqref{divGfp} and~\eqref{divf-fP},
\begin{equation} \label{oddquotients}
\begin{split}
  \divisor \biggl( \frac{(f-f(-P)) G_{\fp}^{(1)}}{t-t_0} \biggr)
  &= (V+P) - (V) + (-V^{(1)} - P) - (-V^{(1)}) \\
  &= \begin{aligned}[t]
    \bigl( &{-(-V^{(1)})} - (V) + (\Xi) + (\infty) \bigr) \\
    &{}+ \bigl( (-V^{(1)} - P) + (V+P) - (\Xi) - (\infty) \bigr).
    \end{aligned}
    \\
  &= \divisor \biggl( \frac{f}{\delta^{(1)}} \cdot \biggl( \frac{ \delta^{(1)}}{f} \biggr)^{\sigma^{-1}} \biggr),
\end{split}
\end{equation}
where the last equality follows from Lemma~\ref{L:Vsigma}(a).  Since both functions have $\tsgn = 1$, they are equal.  Substituting into~\eqref{Lambdafp2} we obtain~\eqref{LambdafpRed}, which concludes the proof.
\end{proof}

\begin{proof}[Proof of Corollary~\ref{C:Lambdaclassno}]
By Theorem~\ref{T:pirho} and its proof we see that $\bigl(\delta^{(1)} \pi_\rho / f \omega_\rho\bigr)\big|_{\Xi} = -1$, and so by Theorem~\ref{T:Lambdavalue1},
\[
  \LL(\bA;1) \big|_{\Xi} = \sum_{\sigma \in \Gal(H/K)} \frac{f}{\delta^{(1)}} \cdot \biggl( \frac{ \delta^{(1)}}{f} \biggr)^{\sigma^{-1}} \bigg|_{\Xi}.
\]
By~\eqref{oddquotients}, since $f(\Xi) = 0$, for each $\sigma \in \Gal(H/K)$ we have
\[
   \frac{f}{\delta^{(1)}} \cdot \biggl( \frac{ \delta^{(1)}}{f} \biggr)^{\sigma^{-1}} \bigg|_{\Xi}
   = -\frac{f(-P) G_{\fp}^{(1)}(\Xi)}{\theta-t_0}
   = \frac{f(-P) f^{\sigma^{-1}}(P)}{\theta-t_0} = 1,
\]
where the second equality follows from Lemma~\ref{L:Gfptwist} and the third from Lemma~\ref{L:Vsigma}(c).
\end{proof}

\section{Proof of a theorem of Anderson} \label{S:LogAlg}

In this section we use the techniques of \S\ref{S:DrinfeldModules}--\ref{S:LSeries} to give a new proof of one of the main results of Anderson~\cite{And94} for our Drinfeld module~$\rho$.

\begin{theorem}[{Anderson~\cite[Thm.~5.1.1]{And94}}] \label{T:Anderson}
Let $B$ be the integral closure of $A$ in $H$.  For any $b \in B$, the power series
\[
  \cE(b;z) := \exp_\rho \Biggl( \sum_{\fa \subseteq A} \frac{b^{\sigma_{\fa}}}{\pd(\rho_{\fa})}\, z^{q^{\deg \fa}} \Biggr) \in \power{H}{z}
\]
is in fact a polynomial in $H[z]$ of degree equal to $q^{q - 1 + \deg b}$.  Moreover, if we let $r = q-1+\deg b$, then we can find $b_1, \dots, b_{r-1} \in H$ so that $\cE(b;z) = bz + b_1z^q + \cdots + b_{r-1}z^{q^{r-1}} + z^{q^r}$.
\end{theorem}

\begin{remark}
Anderson proved that $\cE(b;z)$ is actually an element of $B[z]$, and it would be interesting to see how the methods here can be used to establish the $B$-integrality of $\cE(b;z)$, which is not immediately apparent.  Anderson subsequently proved a more general result in~\cite[Thm.~3]{And96}, which includes an additional twisting parameter in a second variable.  We suspect the methods of this section could be used to approach this result as well, but we do not pursue it here.  Recently Angl\`{e}s, Ngo Doc, and Tavares Ribeiro~\cite{AnglesNgoDacRibeiro17} have given another proof of \cite[Thm.~3]{And96} using class module formulas and Stark units.
\end{remark}

\begin{remark}
Our methods do achieve a slight improvement over Anderson's original theorem.  Anderson proved that the degree in $z$ of $\cE(b;z)$ is bounded above by $q^{q-1 + \max_{\fa}( \deg b^{\sigma_{\fa}})}$, but we find the exact value and show that the leading term of $\cE(b;z)$ is in fact $z^{q^{q-1+\deg b}}$ with coefficient~$1$.
\end{remark}

Originally Theorem~\ref{T:Anderson} was inspired by work of Thakur~\cite{Thakur92} that expressed Goss zeta values for $A$ at $s=1$ in terms of $\log_\rho$ evaluated at algebraic arguments when the class number of $A$ is~$1$.  Later Thakur~\cite[\S 8.10]{Thakur} gave a proof of Anderson's main theorem of~\cite{And96} in the case $A = \F_q[\theta]$, using formulas for power sums and reciprocal sums of polynomials due to Carlitz, Gekeler, and Lee.  Our methods here are similar to Thakur's (who incidentally also encountered obstacles on $\F_q[\theta]$-integrality in some cases).  Broadly, our main tools in proving Theorem~\ref{T:Anderson} are the summation formulas of~\S\ref{S:sums} combined with evaluation properties of the shtuka function from~\S\ref{S:LSeries}.

For $b \in B$, we define $\cL(b;z)$ so that $\cE(b;z) = \exp_\rho(\cL(b;z))$,
\begin{equation} \label{andersonsum}
\cL(b;z) := \sum_{\fa \subseteq A} \frac{b^{\sigma_{\fa}}}{\pd(\rho_{\fa})}\, z^{q^{\deg \fa}} = \sum_{i=0}^\infty c_i z^{q^i} \in \power{H}{z},
\end{equation}
and at the same time fix the sequence of coefficients $\{ c_i \} \subseteq H$.  For a non-zero ideal $\fa \subseteq A$ we pick $\gamma \in A_+$ so that $\gamma \fa^{-1} \subseteq A$, and we set
\begin{equation} \label{cLfasum}
  \cL_{\fa}(b;z) := \frac{\gamma\cdot b^{\sigma_{\fa}^{-1}}}{\pd(\rho_{\gamma \fa^{-1}})} \sum_{a \in \fa_+} \frac{z^{q^{\deg a - \deg \fa}}}{a} \in \power{H}{z}.
\end{equation}

\begin{lemma} \label{L:cLdecomp}
Let $b \in B$.
\begin{enumerate}
\item[(a)] For a non-zero ideal $\fa \subseteq A$, the series $\cL_{\fa}(b;z)$ in \eqref{cLfasum} is independent of the choice of $\gamma \in A_+$ and depends only on the ideal class of $\fa$ in $\Cl(A)$.
\item[(b)] If $\fa_1, \dots, \fa_h \subseteq A$ represent the ideal classes of $A$, then
\[
  \cL(b;z) = \cL_{\fa_1}(b;z) + \cdots + \cL_{\fa_h}(b;z).
\]
\end{enumerate}
\end{lemma}

\begin{proof}
This formulation for $\cL(b;z)$ is due to Lutes~\cite[\S V]{LutesThesis}, and the arguments are similar to the definition of $\Lambda_\fa$ in \eqref{Lambdafa} and the sum for $\LL(\bA;1)$ in \eqref{Lambda1sum}.  The independence from the choice of $\gamma$ is shown using~\eqref{pdrhocfb}: if $\fb = \gamma \fa^{-1}$ and $\fc = \gamma' \fa^{-1}$ are both integral ideals with $\gamma,\,\gamma' \in A_+$, then $\gamma' \fb = \gamma \fc$ and so by \eqref{pdrhocfb}, $\gamma' \cdot \pd(\rho_\fb) = \gamma \cdot \pd(\rho_\fc)$, which yields the desired result.  Showing that $\cL_{\fa}(b;z)$ depends only on the ideal class of $\fa$ involves a similar calculation and we skip the details.  For (b), we note that if $\fc$ is an integral ideal and $c \in A_+$ is chosen so that $\fb := c\fc^{-1} \subseteq A$, then the sum $\cL(b;z)$ from \eqref{andersonsum} taken only over ideals $\fa\neq 0$ equivalent to $\fc$ equals
\begin{equation} \label{cLfb}
  \sum_{\fa \sim \fc} \frac{b^{\sigma_{\fa}}\cdot z^{q^{\deg \fa}}}{\pd(\rho_{\fa})} = \sum_{v \in (\fc^{-1})_+}
  \frac{b^{\sigma_{\fa}} \cdot z^{q^{\deg v + \deg \fc}}}{\pd(\rho_{v\fc})}
  = \sum_{u \in \fb_+} \frac{b^{\sigma_{\fa}}\cdot z^{q^{\deg u-\deg \fb}}}{\pd(\rho_{(u/c)\fc})}.
\end{equation}
Pick $\gamma \in A_+$ so that $\gamma \fb^{-1} \subseteq A$, and so
\[
  \pd\bigl( \rho_{(u/c)\fc} \bigr) = \frac{\pd\bigl( \rho_{(\gamma u/c)\fc} \bigr)}{\gamma}
  = \frac{u\cdot \pd(\rho_{\gamma \fb^{-1}})}{\gamma}.
\]
It follows from this that the sum in~\eqref{cLfb} is $\cL_{\fb}(b;z)$, and then (b) follows from (a).
\end{proof}

\begin{proposition} \label{P:cili}
Let $b \in B$, and recall that $\cL(b;z) = \sum c_i z^{q^i} \in \power{H}{z}$ as defined in~\eqref{andersonsum}.  Then for $i \geq 1$,
\[
  c_i \ell_i = -\frac{f^{(i)}}{\delta^{(i+1)}} \bigg|_{\Xi} \cdot
  \sum_{\sigma \in \Gal(H/K)} b^{\sigma} \cdot f^{\sigma}(V^{(i)}),
\]
where $\ell_i$ is the denominator of the $i$-th coefficient of $\log_\rho$ in~\eqref{liformula}.
\end{proposition}

\begin{proof}
Let $\fp_2, \dots, \fp_h$ be the prime ideals of degree~$1$, as in \eqref{Lambda1sum}.  By Lemma~\ref{L:cLdecomp}(b), we can write $\cL(b;z) = \cL_{(1)}(b;z) + \cL_{\fp_2}(b;z) + \cdots + \cL_{\fp_h}(b;z)$, and by~\eqref{cLfasum}, we see that
\begin{equation} \label{cL1}
  \cL_{(1)}(b;z) = b \cdot \sum_{i=0}^\infty S_i z^{q^i}.
\end{equation}
Combining~\eqref{liformula}, \eqref{lambdai}, and Theorem~\ref{T:Siformula}(a), we have for $i \geq 2$,
\begin{equation} \label{Sili}
  S_i \ell_i = \frac{\nu^{(i)}}{\delta^{(i+1)} \lambda_i^{(1)}} \bigg|_{\Xi}.
\end{equation}
We remark that this quantity is the same as $1/\mu_i$ in Thakur~\cite[Thm.~IV]{Thakur92}.

Likewise, if we temporarily fix $\fp:= \fp_k$ to be any one of the primes of degree~$1$ with corresponding $\F_q$-rational point $P$ and Galois automorphism $\sigma$ and fix $\gamma \in A_+$ such that $\fp^r = (\gamma)$, then by Lemma~\ref{L:gammaquotient} and~\eqref{cLfasum},
\begin{equation} \label{cLfp}
  \cL_{\fp}(b;z) = -b^{\sigma^{-1}} \cdot f(P)^{\sigma^{-1}} \cdot \sum_{i = 0}^\infty
  S_{\fp,i+1}\, z^{q^i}.
\end{equation}
Combining \eqref{liformula}, \eqref{gdeltaquotient}--\eqref{Gkeq}, Theorem~\ref{T:Siformula}(b), and Lemma~\ref{L:Gfptwist}, we find for $i \geq 1$,
\begin{equation} \label{Spili}
  -f(P)^{\sigma^{-1}} \cdot S_{\fp,i+1} \ell_i= -f(P)^{\sigma^{-1}} \cdot \frac{\nu^{(i)}}{\delta^{(i+1)} \lambda_{\fp,i+1}^{(1)} G_{\fp}^{(1)}} \Bigg|_{\Xi}
  = \frac{\nu^{(i)}}{\delta^{(i+1)} \lambda_{\fp,i+1}^{(1)}} \Bigg|_{\Xi}.
\end{equation}
Taking \eqref{cL1}--\eqref{Spili} together, we see that for $i \geq 2$,
\begin{equation} \label{cili1}
  c_i \ell_i = \frac{\nu^{(i)}}{\delta^{(i+1)}} \Biggl( \frac{b}{\lambda_i^{(1)}}
  + \sum_{\fp} \frac{b^{\sigma_{\fp}^{-1}}}{\lambda_{\fp,i+1}^{(1)}} \Biggr) \Bigg|_{\Xi},
\end{equation}
where the sum is over all prime ideals $\fp \subseteq A$ of degree~$1$.  Now, evaluating $\lambda_i^{(1)}$ at $\Xi$, we see from~\eqref{mquotients}, \eqref{fdef}, and~\eqref{lambdai} that for $i \geq 2$,
\[
  \lambda_i^{(1)}(\Xi) = m + \frac{\beta^{q^i} + a_1 \alpha^{q^i} + a_3 + \beta^q}{\alpha^{q^i} - \alpha^q} = m + m^q + a_1 + f^{(1)}(V^{(i)}).
\]
However, using~\eqref{taction} and~\eqref{x1mmq}, we then find that
\begin{equation} \label{lambdaiXi}
  \lambda_i^{(1)}(\Xi) = \frac{t-\theta}{f} \bigg|_{V^{(i)}}.
\end{equation}
Similarly, we see from~\eqref{mquotients}, \eqref{fdef}, and~\eqref{lambdafpi} that for $i \geq 1$,
\[
  \lambda_{\fp,i+1}^{(1)}(\Xi) = m^{\sigma^{-1}} + \bigl( m^{\sigma^{-1}} \bigr)^q + a_1 + \bigl(f^{\sigma^{-1}}\bigr)^{(1)}(V^{(i)}),
\]
and again from \eqref{taction} and~\eqref{x1mmq} we have
\begin{equation} \label{lambdafpiXi}
  \lambda_{\fp,i+1}^{(1)}(\Xi) = \frac{t-\theta}{f^{\sigma^{-1}}} \bigg|_{V^{(i)}}.
\end{equation}
Now $(t-\theta)|_{V^{(i)}} = -\delta^{(i)}(\Xi)$, and so \eqref{cili1}--\eqref{lambdafpiXi} imply that for $i \geq 2$,
\begin{equation} \label{cili2}
c_i\ell_i = -\frac{f^{(i)}}{\delta^{(i+1)}} \bigg|_{\Xi} \cdot \sum_{\sigma \in \Gal(H/K)} b^{\sigma} \cdot f^{\sigma}(V^{(i)}),
\end{equation}
as desired.  It remains to verify the proposition when $i=1$, but this follows from~\eqref{Spili}, \eqref{lambdafpiXi}, and the fact that $f(V^{(1)}) = 0$.
\end{proof}

Consider now the isogeny $1 - \Fr : E \to E$.  If we use coordinates $(t,y)$ on the second copy of $E$ and new coordinates $(\phi,\psi)$ on the first copy of $E$, then we have an induced embedding of fields
\[
  (1-\Fr)^* : \F_q(t,y) \hookrightarrow \F_q(\phi,\psi).
\]
It will be convenient to let $\tE$ be the copy of $E$ in $(\phi,\psi)$-coordinates.
We let $\bH := \F_q(\tE) = \F_q(\phi,\psi)$, and we identify $\bK$ as a subfield of $\bH$ via $(1-\Fr)^*$.  Now $\bH/\bK$ is an abelian unramified extension that is completely split at $\infty$, and so $\bH$ is contained in the Hilbert class field of $\bK$, but since $[\bH:\bK]= \# E(\F_q)$, by maximality $\bH$ is in fact equal to the Hilbert class field of $\bK$.

For $g \in \bK$ and $Q \in \tE$, since $g \in \bH$, we can evaluate $g$ at $Q$.  Since $E$ and $\tE$ are in fact the same, this can lead to confusion, so we will sometimes let $\tg \in \bH$ denote the image of $g$ in $\bH$ via $(1-\Fr)^*$.  Having done this we then have
\begin{equation} \label{tgeval}
  \tg(Q) = g\bigl( (1-\Fr)(Q) \bigr) = g( Q - Q^{(1)}).
\end{equation}
In particular,
\begin{equation} \label{tgevalatV}
  \tg(V) = g(V-V^{(1)}) = g(\Xi).
\end{equation}
{From} this we also see that $\phi$, $\psi$ satisfy the equations in $\bH$,
\[
  \ttt ((\phi,\psi)) = t \bigl( (\phi,\psi) - (\phi,\psi)^{(1)}\bigr) = t,\quad
  \ty((\phi,\psi)) = y \bigl( (\phi,\psi) - (\phi,\psi)^{(1)} \bigr) = y,
\]
where the left-hand sides are rational functions in $\phi$, $\psi$ given by the first and second coordinates of the point $(\phi,\psi) - (\phi,\psi)^{(1)} \in E$, whereas the right-hand sides are the functions $t$, $y \in \bK \subseteq \bH$.  Thus if we extend the evaluation isomorphism $\iota : \bK \iso K$ to $\iota : \bH \iso H$, we see that $(\iota(\phi),\iota(\psi)) \in E(H)$ satisfies~\eqref{Vdef}.  In this way we choose an isomorphism $\iota : \bH \iso H$ so that $(\iota(\phi),\iota(\psi)) = V$, i.e.,
\[
  \iota(\phi) = \phi(V) = \alpha, \quad \iota(\psi) =\psi(V) = \beta.
\]
We denote the inverse of $\iota$ by $\chi : H \iso \bH$, which is an extension of the canonical isomorphism $\chi : K \iso \bK$.  As previously, for $c \in H$ we will also denote $\chi(c) = \oc$.  There is then a natural isomorphism
\begin{equation}
  \sigma \mapsto \osigma : \Gal(H/K) \to \Gal(\bH/\bK)
\end{equation}
for which
\[
  \overline{\sigma(c)} = \osigma(\oc), \quad c \in H,
\]
thus making the two bar operations compatible.  Furthermore,
\begin{equation} \label{alphabetabar}
\phi = \oalpha, \quad \psi = \obeta.
\end{equation}
For $\sigma \in \Gal(H/K)$, we can pick $P_\sigma \in E(\F_q)$ so that as in Lemma~\ref{L:Vsigma} we have $V^{\sigma} = V - P_{\sigma}$, and moreover we find that
\[
  \phi^{\osigma} = \overline{t(V-P_{\sigma})}, \quad \psi^{\osigma} = \overline{y(V-P_{\sigma})}.
\]
Thus on $\tE$ we have the identity,
\begin{equation} \label{phipsiosigma}
  \bigl(\phi^{\osigma}, \psi^{\osigma}\bigr) = (\phi,\psi) - P_{\sigma}, \quad \sigma \in \Gal(H/K),
\end{equation}
and in particular for $V \in \tE$,
\[
  \phi^{\osigma}(V) = t(V-P_{\sigma}) = \alpha^{\sigma}, \quad
  \psi^{\osigma}(V) = y(V-P_{\sigma}) = \beta^{\sigma}.
\]
If we consider the field $H(\tE) = H(\phi,\psi) = \F_q(\theta,\eta,\alpha,\beta,t,y,\phi,\psi)$ (the compositum of $H$ and $\bH$), for $g \in H(t,y)$ we can define $\tg \in H(\phi,\psi)$ as above with the same meaning as in \eqref{tgeval}.

We now fix the following function in $H(\phi,\psi)$,
\begin{equation} \label{Fdef}
  \cF := \frac{\beta + \psi + a_1 \phi + a_3}{\alpha - \phi} -
  \frac{\psi^q + \psi + a_1 \phi + a_3}{\phi^q - \phi}.
\end{equation}
Returning to the sum in Proposition~\ref{P:cili}, we see that for $\sigma \in \Gal(H/K)$ and $i \geq 1$,
\[
  f^{\sigma}\bigl(V^{(i)} \bigr) = \bigl( \cF^{\osigma} \bigr)^{(i)}(V),
\]
where $\osigma$ acts only on elements of $\bH$ leaving elements from $H$ fixed.  We note that in this equation, $V^{(i)}$ on the left is in $E$ whereas $V$ on the right is in $\tE$.  For fixed $b \in B$ if we let
\begin{equation} \label{gbdef}
  g_b := \sum_{\sigma \in \Gal(H/K)} (\ob\cdot \cF)^{\osigma} \quad \in H(t,y),
\end{equation}
then by \eqref{tgevalatV},
\begin{equation} \label{E:gbtwist}
  \sum_{\sigma \in \Gal(H/K)} b^{\sigma} \cdot f^{\sigma} \bigl( V^{(i)} \bigr)
  = \tg_b^{(i)}(V) = g_b^{(i)}(\Xi).
\end{equation}
Our intermediate goal is to determine information about the divisor of $g_b$ as a function on~$E$ with respect to our original $(t,y)$-coordinates.

\begin{lemma} \label{L:kappaR}
Let $L/\F_q$ be an algebraically closed field, and let $R \in E(L) \setminus E(\overline{\F}_q)$.  Let
\[
  \kappa_R = \frac{y(R) + y + a_1t + a_3}{t(R)-t} - \frac{y^q + y+a_1t+a_3}{t^q-t} \in L(t,y),
\]
and fix $Q_0 \in E(L)$ with $Q_0 - Q_0^{(1)} = R$.  Then
\[
  \divisor(\kappa_R) = (Q_0) + q\bigl(R^{(-1)} \bigr) - (R) - q(\infty) + \sum_{P \in E(\F_q), \, P \neq \infty}  \bigl( (Q_0+P) - (P) \bigr).
\]
\end{lemma}

\begin{proof}
We first consider the poles of $\kappa_R$.  Let $\kappa_1 = (y(R) + y + a_1t + a_3)/(t(R)-t)$, and let $\kappa_2 = (y^q + y + a_1t + a_3)/(t^q-t)$.  Away from $\infty$, $\kappa_1$ has a simple pole at $R$ and no other poles, and $\kappa_2$ has simple poles at each point in $E(\F_q) \setminus \{ \infty\}$ and no other poles.  The degree of $\kappa_R$ is $q$, and since $R \notin E(\F_q)$, we see that the divisor of poles of $\kappa_R$ is as desired.

For zeros, certainly $\kappa_R$ vanishes at $R^{(-1)}$.  Combining the two terms of $\kappa_R$ together and adding and substracting $t(R)y(R)$ in the numerator, we obtain
\[
  \kappa_R = \frac{t(y^q - y(R)) + (y+a_1t+a_3)(t^q-t(R)) + y(R)(t^q-t(R)) - t(R)(y^q-y(R))}{(t^q-t)(t(R)-t)},
\]
from which it follows that the order of vanishing of $\kappa_R$ at $R^{(-1)}$ is at least $q$.
If we let $Q = (t,y)$, then $\kappa_1$ is the slope between $R$ and $-Q$ and $\kappa_2$ is the slope between $Q^{(1)}$ and $-Q$.  If $Q \neq R^{(-1)}$, then $\kappa_R$ vanishes when $Q^{(1)}$, $-Q$, and $R$ are collinear, i.e., when $R = Q - Q^{(1)}$.  Fix any $Q_0 \in E(L)$ satisfying $Q_0 - Q_0^{(1)} = R$. Since $R \notin E(\overline{\F}_q)$, we have that $Q_0 \notin E(\overline{\F}_q)$ and that the three points $R$, $-Q_0$, and $Q_0^{(1)}$ must be distinct. We see that $\kappa_R$ has zeros at each of the distinct points $Q_0 + P$, $P \in E(\F_q)$.  Since the degree of the polar divisor is $-q-\#E(F_q)$, we see that the order of vanishing at $R^{(-1)}$ is exactly $q$ and that each zero at $Q_0 + P$ is simple as desired.
\end{proof}

\begin{proposition} \label{P:gb}
For $b \in B$, let $g_b \in H(E) = H(t,y)$ be given as in \eqref{gbdef}.
\begin{enumerate}
\item[(a)] The divisor of poles of $g_b$ is precisely $-(\Xi) - (q+\deg b)(\infty)$.
\item[(b)] $g_b(V) = 0$.
\item[(c)] We have $f\cdot g_b \in N = \Gamma(U,\cO_E(-(V^{(1)})))$, the dual $\bA$-motive of $\rho$.
\end{enumerate}
\end{proposition}

\begin{proof}
The function $\cF \in H(\phi,\psi)$ from \eqref{Fdef} is the same as $\kappa_V$ from Lemma~\ref{L:kappaR} with respect to $(\phi,\psi)$-coordinates and with $R=V$.  Letting $\divisor_\infty(\cF)$ denote the polar divisor of $\cF$, Lemma~\ref{L:kappaR} implies
\[
  \divisor_\infty(\cF) = -(V) - (q-1)(\infty) - \sum_{P \in E(\F_q)} (P).
\]
For $\sigma \in \Gal(H/K)$, it follows from~\eqref{phipsiosigma} and~\eqref{Fdef} that $\cF^{\osigma} = \kappa_V(X - P_{\sigma})$, i.e., the pullback of $\kappa_V$ by translation by~$-P_{\sigma}$.  Therefore by Lemma~\ref{L:kappaR},
\[
  \divisor_\infty(\cF^{\osigma}) = -(V+P_{\sigma}) - (q-1)(P_{\sigma}) - \sum_{P \in E(\F_q)} (P).
\]
Now for $b \in B$, we see that
\[
  \divisor_\infty(\tg_b) = \divisor_\infty \Biggl( \sum_{\sigma \in \Gal(H/K)} (\ob\cdot  \cF)^{\osigma}\Biggr)
  \geq -\sum_{P \in E(\F_q)} \bigl( (V+P) + (q+\deg b)(P) \bigr).
\]
Since $\tg_b = (1-\Fr)^*(g_b)$, we see that on $E$ with respect to $(t,y)$-coordinates,
\[
  \divisor_\infty(g_b) \geq -(1-\Fr)(V) - (q+\deg b)((1-\Fr)(\infty)) = -(\Xi) - (q+\deg b)(\infty).
\]
To prove part (a) we need to show that this is an equality.  However, $(\ob \cdot \cF)^{\osigma}$ has a simple pole at $V+P_{\sigma}$ and is regular at $V+P$ for all $P \in E(\F_q) \setminus \{P_{\sigma} \}$.  Therefore the sum over all $\sigma \in \Gal(H/K)$ has exactly simple poles at each of $V+P$, $P \in E(\F_q)$, and so when descending to $(t,y)$-coordinates, we see that $g_b$ must have a simple pole at $(1-\Fr)(V) = \Xi$.  By a similar argument we find that $g_b$ has a pole of order $q + \deg b$ at $\infty$.

To prove part (b) we fix $Q_0 \in \tE$ with $Q_0 - Q_0^{(1)} = V$, and we note by Lemma~\ref{L:kappaR} that $\cF$ vanishes at each $Q_0+P$, $P \in E(\F_q)$.  As in the previous paragraph we find that for each $\sigma \in \Gal(H/K)$, the function $\cF^{\osigma}$ also vanishes at each $Q_0+P$, $P \in E(\F_q)$.  Therefore, $\tg_b = \sum_\sigma (\ob \cdot F)^{\osigma}$ vanishes at these same translates of $Q_0$.  It follows that with respect to $(t,y)$-coordinates, $g_b$ vanishes at $(1-\Fr)(Q_0) = V$.  Part (c) then follows from (a) and~(b).
\end{proof}

\begin{proof}[Proof of Theorem~\ref{T:Anderson}]
By Proposition~\ref{P:cili} and \eqref{E:gbtwist}, if we let $\cL(b;z) = \sum c_i z^{q^i} \in \power{H}{z}$, then for $i \geq 1$,
\[
  c_i \ell_i = -\biggl( \frac{f \cdot g_b}{\delta^{(1)}} \biggr)^{(i)} \Bigg|_{\Xi}.
\]
By Proposition~\ref{P:gb}, we know that $f \cdot g_b \in N$ and that $\ord_\infty(f \cdot g_b) = -(q+1 + \deg b)$.  Thus letting $r = q-1+\deg b$, by \eqref{dualbasis} we can find $e_0, \dots, e_{r} \in \oK$ so that
\begin{equation} \label{fgbdecomp}
  -f \cdot g_b = e_0 \delta^{(1)} + e_1 \delta f + e_2 \delta^{(-1)} f f^{(-1)} + \cdots
  + e_{r} \delta^{(-r+1)} f f^{(-1)} \cdots f^{(-r+1)},
\end{equation}
and $e_{r} \neq 0$.  We calculate $e_0$ by evaluating at $\Xi$ and find $e_0 = -(f\cdot g_b/\delta^{(1)}) |_{\Xi}$.  We note by \eqref{tgevalatV} that $(f\cdot g_b)|_{\Xi} = (\tf\cdot \tg_b)|_V$, and on the other hand by \eqref{Fdef} and~\eqref{gbdef},
\begin{equation} \label{fgbsum}
  -f\cdot g_b = \tf \cdot \sum_{\sigma \in \Gal(H/K)} \ob^{\osigma} \biggl(
  \frac{\psi^{\osigma} + a_1\phi^{\osigma} + a_3 + \beta}{\phi^{\osigma} - \alpha} + \frac{(\psi^{\osigma})^q + \psi^{\osigma} + a_1\phi^{\osigma} + a_3}{(\phi^{\osigma})^q-\phi^{\osigma}} \biggr),
\end{equation}
as elements of $H(\phi,\psi)$.  Since $\tf(V) = 0$ and $\phi^{\osigma}(V) = \alpha^{\sigma}$, we see that only a single term in this sum is non-zero when we evaluate at $V$, namely
\[
  -(f\cdot g_b) \big|_{\Xi} = (\tf\cdot \tg_b)|_V = \biggl( \tf \cdot \ob \cdot \frac{\psi + a_1 \phi + a_3 + \beta}{\phi-\alpha} \biggr)\bigg|_{V}.
\]
Combining with \eqref{x1ftwist}--\eqref{alphaidentity}, we have
\begin{align*}
  e_0 = -\biggl( \frac{f \cdot g_b}{\delta^{(1)}} \biggr) \bigg|_{\Xi} &= \frac{b(2\beta + a_1\alpha+a_3)}{\delta^{(1)}(\Xi)} \cdot \biggl( \frac{f}{t-\theta} \biggr) \bigg|_{\Xi}
  \cdot \biggl( \frac{\widetilde{t-\theta}}{\phi-\alpha} \biggr) \bigg|_{V} \\
  &= \frac{b(2\beta + a_1\alpha+a_3)}{2\eta + a_1\theta+a_3} \cdot \biggl( \frac{\ttt-\theta}{\phi-\alpha} \biggr) \bigg|_{V} \\
  &= \frac{b(2\beta + a_1\alpha+a_3)}{2\eta + a_1\theta+a_3} \cdot \biggl( \frac{d\ttt}{d\phi} \biggr) \bigg|_{V}.
\end{align*}
Now $(1-\Fr)^* (dt/(2y+a_1t+a_3)) = d\phi/(2\psi + a_1\phi+a_3)$, and so
\[
  \biggl( \frac{d\ttt}{d\phi} \biggr) \bigg|_{V} = \frac{2\eta + a_1 \theta +a_3}{2\beta + a_1\alpha+a_3}.
\]
Thus $e_0 = b$, and using~\eqref{andersonsum} we see that $c_0 = b$.  Combining this with Proposition~\ref{P:cili} and~\eqref{gbdef}, we have
\begin{align*}
  \cL(b;z) = \sum_{i=0}^\infty c_i z^{q^i} &= b z - \sum_{i=1}^\infty \frac{1}{\ell_i}
  \biggl( \frac{f \cdot g_b}{\delta^{(1)}} \biggr)^{(i)} \Bigg|_{\Xi} \cdot z^{q^i} \\
  &= b z + \sum_{i=1}^\infty \frac{1}{\ell_i \delta^{(i+1)}(\Xi)} \Biggl( b \delta^{(1)}
  + \sum_{j=1}^{r} e_j \delta^{(-j+1)} f f^{(-1)} \cdots f^{(-j+1)} \Biggr)^{(i)} \Bigg|_{\Xi} \cdot z^{q^i}.
\end{align*}
By~\eqref{liformula} (and using that $f(\Xi)=0$), we obtain
\begin{align*}
\cL(b;z) &= b z + \sum_{i=1}^\infty \frac{1}{\delta^{(1)} f^{(1)} \cdots f^{(i)} |_{\Xi}}
\Biggl( b^{q^i} \delta^{(i+1)}(\Xi) + \sum_{j=1}^r e_j^{q^i} \delta^{(i-j+1)} f^{(i)} \cdots f^{(i-j+1)} \Biggr) \Bigg|_{\Xi}\cdot z^{q^i} \\
&= bz + \sum_{i=1}^\infty \Biggl( \frac{b^{q^i} z^{q^i}}{\ell_i} + \sum_{j=1}^{\min(i,r)}
\frac{\bigl(e_j^{q^j} z^{q^j}\bigr)^{q^{i-j}}}{\ell_{i-j}} \Biggr) \\
&= \log_\rho \bigl( bz + e_1 z^q + \cdots + e_r^{q^r} z^{q^r} \bigr).
\end{align*}
Using~\eqref{fgbsum} we can show that $\tsgn(-f\cdot g_b) = 1$, and so by~\eqref{fgbdecomp} we must have $e_r = 1$.  Therefore
\begin{equation} \label{Ebzdecomp}
\cE(b;z) =  bz + e_1 z^q + \cdots + e_{r-1}^{q^{r-1}} z^{q^{r-1}} + z^{q^r},
\end{equation}
as desired.  We note that each $e_j^{q^j} \in H$, since a priori $\cE(b;z) \in \power{H}{z}$.
\end{proof}

\begin{remark} \label{Rem:Ebz}
By the proof above, we can calculate $\cE(b;z)$ explicitly via the decomposition of $-f\cdot g_b$ as an element of the dual $\bA$-motive $N$ as in~\eqref{fgbdecomp}.  See \S\ref{S:Examples} for examples.
\end{remark}

\section{Examples} \label{S:Examples}

\begin{example}[Class Number $1$]\label{ex1}
We take $E_1: y^2 = t^3 - t - 1$ over $\F_3$ so that $E_1(\F_3) = \{ \infty\}$.  Much of this example has been worked out by Thakur \cite[Thm. VI]{Thakur92}, \cite[\S 2.3(c)]{Thakur93}.  Thakur shows in the notation of \S\ref{S:DrinfeldModules} that $V = (\theta+1,\eta)$, $m = \eta$, and $x_1 = \eta^3+\eta$, and moreover,
\[
  f = \frac{y - \eta -\eta(t-\theta)}{t - \theta - 1}.
\]
If we fix $\sqrt{-1} \in \F_9$ and let $Q = (0,\sqrt{-1}) \in E(\F_9)$, then we can define a Dirichlet character $\chi_Q : A \to \F_9$ given by $\chi_Q(a) = a(Q)$.  We form the Dirichlet $L$-value,
\[
  L(\chi_Q,1) = \sum_{a \in A_+} \frac{\chi_Q(a)}{a} = L(\bA;1)\big|_{Q}.
\]
The quantity $\xi$ in~\eqref{omegarhoprod} and Theorem~\ref{T:pirho} is $\xi = -(m + \beta/\alpha) = -\eta(\theta-1)/(\theta+1)$, and evaluating the identity in Theorem~\ref{T:Lvalue1} at $Q$, we find
\[
  L(\chi_Q,1)
  = - \frac{\sqrt{-1}(\theta^9+1)^{1/2}}{\eta^{3/2}(\theta^3-1)^{1/2}} \cdot  \biggl( 1 + \dfrac{\sqrt{-1}}{\eta^3(\theta^3-1)} \biggr)^{-1/8} \cdot \biggl(1 + \dfrac{\sqrt{-1}}{\eta^9(\theta^9-1)} \biggr)^{-1/8} \cdot \pi_\rho,
\]
thus providing a new formula from the one in \cite[Ex.~VIII.4]{LutesThesis}.

As an example of the construction in Theorem~\ref{T:Anderson}, following the exposition in \S\ref{S:LogAlg} we have $\bH = \F_3(\phi,\psi) = \bK = \F_3(t,y)$, since $\#E(\F_3) = 1$, and we find that
\[
  \phi = \overline{\theta}+1 =  t+1, \quad \psi = \overline{\eta} = y.
\]
We then have from~\eqref{Fdef}
\[
\cF = \frac{\eta + y}{\theta - t} - \frac{y^3+y}{t^3-t} = \frac{\eta + y}{\theta -t} -y.
\]
One verifies that by taking $b=1$ in \eqref{gbdef},
\[
  g_1 = -f\cdot \cF =  \delta^{(1)} + \eta^{1/3}\delta f + \delta^{(-1)} ff^{(-1)},
\]
which by \eqref{Ebzdecomp} then implies that as power series in $\power{K}{z}$,
\[
  \sum_{a \in A_+} \frac{z^{3^{\deg a}}}{a} = \log_\rho \bigl( z + \eta z^3 + z^9 \bigr),
\]
thus confirming a formula of Anderson~\cite[p.~492]{And94}.
\end{example}

\begin{example}[Class Number $2$]\label{ex2}
We take $E_2 : y^2 = t^3 -t^2 - t$ over $\F_3$ so that $E_2(\F_3) = \{ (0,0), \infty\}$ (cf.\ Hayes~\cite[Ex.~11.7]{Hayes79}, Lutes~\cite[Exs.~VIII.5--6]{LutesThesis}).  We let $P=(0,0)$ and $\fp = (\theta,\eta)$.  We find that $H = K(\sqrt{\theta})$, that $\fp^2 = (\theta)$, and that
\[
  V = (\alpha,\beta) = \biggl( -\theta - 1 - \frac{\eta}{\sqrt{\theta}},
  -\eta - \theta\sqrt{\theta} - \sqrt{\theta} \biggr).
\]
Of some use is that $\alpha$ is a fundamental unit for $B \subseteq H$, and in spite of appearances one checks that $\sgn(\alpha)=\sgn(\beta)=1$ as in~\eqref{sgnalphabeta}.  We calculate that $m = -\eta - \theta\sqrt{\theta} + \sqrt{\theta}$, and so
\[
  f = \frac{y - \eta - (-\eta - \theta\sqrt{\theta} + \sqrt{\theta})(t-\theta)}{t  +\theta + 1 + (\eta/\sqrt{\theta})}.
\]
We find that $f(P) = \xi$ in \eqref{omegarhoprod}, and moreover that $\xi = \eta + \theta\sqrt{\theta} + \theta = -\beta = -\sqrt{\theta} \alpha$, and thus $\rho_\fp = \tau + \sqrt{\theta}\alpha$.  We also calculate that the coefficient $x_1$ in $\rho_t$ from~\eqref{taction} is
\[
x_1 = \sqrt{\theta} - \theta^4\cdot\sqrt{\theta} - \eta - \eta^3.
\]
We let $\chi_P : A \to \F_3$ be the Dirichlet character defined by $\chi_P(a) = a(P)$, which peels off the constant term of $a$, and evaluating the identity in Theorem~\ref{T:Lvalue1} at $P$ we find
\[
  L(\chi_P,1) = \sum_{a\in A_+} \frac{\chi_P(a)}{a} = L(\bA;1)\big|_{P}
  = \frac{\sqrt{-1} \cdot \alpha^{3/2}}{\theta^{3/4}} \cdot \pi_\rho,
\]
noting that $\omega_\rho(P) = (-\sqrt{\theta} \alpha)^{1/2}$.  (The choice of $\sqrt{-1}$ is made to be consistent with the one taken in~\eqref{omegarhoprod}.)  By way of an application of Theorem~\ref{T:Lambdavalue1}, we calculate
\begin{equation} \label{LambdaevalP}
  \LL(\bA;1)\big|_{P} = \frac{\sqrt{-1} \cdot (\alpha^2 - (\alpha^{\sigma})^2)}{\theta^{3/4} \cdot \sqrt{\alpha}} \cdot \pi_\rho = \frac{\sqrt{-1}\cdot \eta (\theta+1)}{\theta^{5/4} \cdot \sqrt{\alpha}} \cdot \pi_\rho.
\end{equation}
Then taking $\gamma = \theta$ for $\Lambda_{\fp}$ in \eqref{Lambdafa}, we obtain
\[
  \Lambda_{\fp} = \frac{\chi(\fp)/t}{\pd(\rho_{\fp})/\theta} \sum_{a \in \fp_+} \frac{a(t,y)}{a(\theta,\eta)}
  = \frac{\sqrt{\theta}}{\alpha} \cdot \frac{f - f(P)}{t} \cdot \sum_{a \in \fp_+} \frac{a(t,y)}{a(\theta,\eta)}.
\]
We find that if $a \in \fp$, then
\[
  \frac{(f-f(P)) \cdot a(t,y)}{t} \bigg|_{P} = \frac{(a(t,y)/y)(P)}{\alpha},
\]
where if we write $a(t,y) = b(t) + c(t)y$ with $b$, $c \in \F_3[t]$, then $(a(t,y)/y)(P) = c(0)$.  Thus if we define $\tchi_P : \fp \to \F_3$ by $\tchi_P(a) = (a(t,y)/y)(P)$ with corresponding Dirichlet series $L(\tchi_P,s)$, then by~\eqref{Lambda1sum} and~\eqref{LambdaevalP},
\begin{equation}
  \LL(\bA;1)\big|_P = L(\chi_P,1) + \frac{\sqrt{\theta}}{\alpha^2}\cdot L(\tchi_P,1)
  = \frac{\sqrt{-1}\cdot \eta (\theta+1)}{\theta^{5/4} \cdot \sqrt{\alpha}} \cdot \pi_\rho.
\end{equation}

For the construction in Theorem~\ref{T:Anderson}, we observe that $H(\phi,\psi) = \F_3(\sqrt{\theta},\eta,\sqrt{t},y)$.  We see that $\phi$, $\psi$ satisfy the equations $\phi^2 - (t+1)\phi - 1 = 0$ and $\psi^2 - t\phi^2 = 0$. Moreover,
\begin{align*}
\phi &= \oalpha = -t-1-\frac{y}{\sqrt{t}} & \psi &= \obeta = -y - t\sqrt{t} - \sqrt{t}, \\
\phi^{\osigma} &= \overline{\alpha^{\sigma}} = -t-1+ \frac{y}{\sqrt{t}} & \psi^{\osigma} &= \overline{\beta^{\sigma}} = -y + t \sqrt{t} + \sqrt{t},
\end{align*}
where $\sigma$ is the element of order $2$ in $\Gal(H/K)$.  Then using $b=1$ in \eqref{gbdef} one computes
\[
  -f \cdot g_1 = -f \cdot (\cF + \cF^{\osigma}) = \delta^{(1)} + m^{1/3}\delta f + \delta^{(-1)} f f^{(-1)},
\]
which combined with~\eqref{Ebzdecomp} gives
\begin{equation} \label{Ex2log1}
  \sum_{\fa \subseteq A} \frac{z^{q^{\deg \fa}}}{\pd(\rho_{\fa})}
  = \log_\rho(z + mz^3 + z^9).
\end{equation}
On the other hand, using $b = \sqrt{\theta}$ in \eqref{gbdef}, we have $g_{\sqrt{\theta}} = \sqrt{t}\cdot \cF - \sqrt{t}\cdot \cF^{\osigma}$, and we verify
\[
  -f\cdot g_{\sqrt{\theta}} =
 \delta^{(1)} - \biggl( \theta - 1 - \frac{\alpha \eta^2}{\theta} \biggr)^{1/3} \delta f
  + x_1^{1/9} \delta^{(-1)} f f^{(-1)} + \delta^{(-2)} f f^{(-1)}f^{(-2)}.
\]
If for an ideal $\fa \subseteq A$ we let $\epsilon(\fa) = \pm 1$ depending on whether $\fa$ is or is not principal, then \eqref{Ebzdecomp} yields the identity in $\power{H}{z}$,
\begin{equation} \label{Ex2log2}
   \sum_{\fa \subseteq A} \frac{\epsilon(\fa) \sqrt{\theta}\cdot z^{q^{\deg \fa}}}{\pd(\rho_{\fa})} = \log_{\rho} \biggl( \sqrt{\theta} z - \biggl( \theta - 1 - \frac{\alpha\eta^2}{\theta} \biggr) z^3 + x_1 z^9 + z^{27} \biggr).
\end{equation}
We note that \eqref{Ex2log1} and \eqref{Ex2log2} agree with previous calculations of Lutes~\cite[p.~125]{LutesThesis}.
\end{example}

\end{document}